\documentclass{article}
\usepackage{titlesec}
\usepackage{titling}
\titleformat{\chapter}{\normalfont\huge}{\thechapter.}{20pt}{\huge\it}
\renewcommand{\paragraph}{\subsection}
\usepackage[utf8]{inputenc}
\usepackage[T1]{fontenc}
\usepackage{amsmath,amssymb,amsthm,mathrsfs,calligra,amsfonts}
\usepackage{enumerate}
\usepackage[framemethod=tikz]{mdframed}
\usepackage{array}
\usepackage{lipsum}
\usepackage{mathtools}
\usepackage{scrextend}
\usepackage{tikz-cd}
\usepackage{cite}
\usepackage{pdfrender,xcolor}
\usepackage{lmodern}
\usepackage[hidelinks]{hyperref}
\usepackage[noabbrev,capitalize]{cleveref}
\theoremstyle{plain}
\newtheorem{theorem}{Theorem}[subsection]   
\newtheorem{corollary}[theorem]{Corollary}
\newtheorem{lemma}[theorem]{Lemma}

\newtheorem{proposition}[theorem]{Proposition}

\theoremstyle{definition}
\newtheorem{definition}[theorem]{Definition}

\newtheorem{example}[theorem]{Example}

\theoremstyle{remark}
\newtheorem{remark}[theorem]{Remark}
\titleformat*{\subsection}{\itshape\mdseries}
\crefformat{section}{\S#2#1#3}

\newcommand{\mr}{\mathrm}

\newcommand{\Aa}{\mr{A}_{\infty}}

\newcommand{\C}{\mathbf C}

\newcommand{\brak}{[\![\hbar]\!]}
\newcommand{\cbrak}{(\!(\hbar)\!)}

\newcommand{\2}{\vspace{2mm}}
\newcommand{\1}{\vspace{1mm}}

\newcommand*{\shom}{\mathscr{H}\text{\kern -3pt {\it{om}}}\,}
\newcommand*{\ext}{\mathscr{E}\text{\kern -1.5pt {\it{xt}}}\,}
\makeatletter
\newcommand{\xdashrightarrow}[2][]{\ext@arrow 0359\rightarrowfill@@{#1}{#2}}
\newcommand{\xdashleftarrow}[2][]{\ext@arrow 3095\leftarrowfill@@{#1}{#2}}
\newcommand{\xdashleftrightarrow}[2][]{\ext@arrow 3359\leftrightarrowfill@@{#1}{#2}}
\def\rightarrowfill@@{\arrowfill@@\relax\relbar\rightarrow}
\def\leftarrowfill@@{\arrowfill@@\leftarrow\relbar\relax}
\def\leftrightarrowfill@@{\arrowfill@@\leftarrow\relbar\rightarrow}
\def\arrowfill@@#1#2#3#4{%
  $\m@th\thickmuskip0mu\medmuskip\thickmuskip\thinmuskip\thickmuskip
   \relax#4#1
   \xleaders\hbox{$#4#2$}\hfill
   #3$%
}
\makeatother
\usepackage{geometry}               
\usepackage[parfill]{parskip}    
\usepackage{graphicx}
\usepackage{epstopdf}
\title{\vspace{-1cm}Formality of differential graded algebras and complex Lagrangian submanifolds}
\author{Borislav Mladenov}
\AtEndDocument{\bigskip{\footnotesize%
  \text{Department of Mathematics, Imperial College London, London SW7 2AZ, United Kingdom} \\ 
  \href{mailto:bm1514@ic.ac.uk}{\textit{bm1514@ic.ac.uk}}
  }}
\predate{}
\date{} 
\postdate{}
\begin{document}
\maketitle
\nocite{*}
\begin{abstract}
 Let $i: \mr{L} \xhookrightarrow{} \mr{X}$ be a compact K\"{a}hler Lagrangian in a holomorphic symplectic variety $\mr{X}/\C$. We use deformation quantisation to show that the endomorphism differential graded algebra $\mr{RHom}\big(i_*\mr{K}_\mr{L}^{1/2},i_*\mr{K}_\mr{L}^{1/2}\big)$ is formal. We prove a generalisation to pairs of Lagrangians, along with auxiliary results on the behaviour of formality in families of $\Aa$-modules.
\end{abstract}

\paragraph{Introduction.}

Let $i : \mr{L}\xhookrightarrow{} \mr{X}$ be a smooth compact Kähler Lagrangian in a holomorphic symplectic (not necessarily proper!) variety over $\C$. Recall that, for any $\mathscr{L} \in \mr{Pic}(\mr{L})$, there exists a local-to-global $\mr{Ext}$ spectral sequence: 
\begin{equation*}\mr{E}_{2}^{p,q}=\mr{H}^{p}(\mr{L},\Omega^{q}_\mr{L}) \Rightarrow \mr{Ext}^{p+q}(i_{*}\mathscr{L},i_{*}\mathscr{L}).\end{equation*}Therefore the second page gives de Rham cohomology of $\mr{L}$, independent of $\mathscr{L}$, but the differentials depend on $\mathscr{L}$. In \cite{bm1} we show that they vanish if $\mathscr{L}$ is a rational power of $\mr{K}_{\mr{L}}$ such as $\mathscr{O}_\mr{L}$.
\2
\begin{theorem}\label{thmpaper1}
 Let $i : \rm L \xhookrightarrow{} X$ be a smooth complex Lagrangian in a compact hyperkähler variety $\mr{X}/\C$, and let $\mathscr{L}$ be any (existing) rational power of the canonical bundle of $\mr{L}$. Then the local-to-global $\mr{Ext}$ spectral sequence  $$\mr{E}_{2}^{p,q}= \mr{H}^{p}(\mr{L},\Omega^{q}_\mr{L}) \Rightarrow \mathrm{Ext}^{p+q}(i_{*}\mathscr{L},i_{*}\mathscr{L})$$ degenerates (multiplicatively) on the second page.
\end{theorem}
\2
 $\hspace*{3mm}$The proof, operating purely within algebraic geometry, is a Deligne-stlye argument based on the Lefschetz operator induced by the K\"{a}hler form. Here we give a new proof in the special case $\mathscr{L} = \mr{K}_\mr{L}^{1/2}$ using DQ modules. This doesn't require $\mr{X}$ compact Kähler and, in fact, the proof shows we don't even need $\mr{L}$ Kähler - it is enough it be smooth, compact with Hodge-to-de Rham spectral sequence degenerating on $\mr{E}_1$.\\
 $\hspace*{3mm}$Recall that a differential graded algebra $\mr{A}$ is formal if $\mr{A}\simeq \mr{HA}$ as differential graded algebras. By Kadeishvili's theorem the cohomology of a differential graded algebra is naturally an $\Aa$-algebra. The $\Aa$-structure on $\mr{HA}$ measures failure of formality for $\mr{A}$. Kähler formality of Deligne et al. (see \cite{Deligne1975}) says that $\Omega^{*}(\mr{L},\C)$ is a formal differential graded algebra, hence $\mr{H}^{*}(\mr{L}/\C)$ is formal as an $\Aa$-algebra. In light of $\mr{H}^{k}(\mr{L}/\C) = \oplus_{p,q}\mr{H}^{p}(\mr{L},\Omega^{q}_\mr{L}) = \mr{Ext}^k(i_{*}\mathscr{L},i_{*}\mathscr{L})$, it is natural to ask whether the same goes for $\mr{Ext}(i_*\mathscr{L},i_*\mathscr{L})$, i.e. if the differential graded algebra $\mr{RHom}(i_*\mathscr{L},i_*\mathscr{L})$ is formal. We confirm this in the case where $\mathscr{L}$ is a square root of the canonical bundle of $\mr{L}$.
 \2
 \begin{theorem}\label{thm01}
  Let $\mr{X}/\C$ be holomorphic symplectic and let $i:\mr{L \xhookrightarrow{} X}$ be a smooth compact Kähler Lagrangian submanifold whose canonical bundle admits a square root. Then the differential graded algebra $\mr{RHom}\big(i_{*}\mr{K}_\mr{L}^{1/2},i_{*}\mr{K}_\mr{L}^{1/2}\big)$ is formal, in fact, quasi-isomorphic to the de Rham algebra $\mr{H}(\mr{L}/\C)$.
 \end{theorem}
\2
 $\hspace*{3mm}$Consider two smooth Lagrangians $i:\mr{L}\xhookrightarrow{} \mr{X}$, $j: \mr{M}\xhookrightarrow{} \mr{X}$ such that $\mr{L}\cap\mr{M}$ is smooth. For any choice of line bundles $\mathscr{L} \in \mr{Pic}(\mr{L})$ and $\mathscr{M} \in \mr{Pic}(\mr{M})$, we have a local-to-global $\mr{Ext}$ spectral sequence
 \begin{equation*}\mathrm{E}_{2}^{p,q}=\mr{H}^{p}(\mr{L}\cap \mr{M},\Omega^{q-c}_{\mr{L}\cap \mr{M}} \otimes \mathscr{K})\Rightarrow \mathrm{Ext}^{p+q}(i_{*}\mathscr{L},j_{*}\mathscr{M}),\end{equation*} where $\mathscr{K} \coloneqq \mr{det \mathscr{N}_{L\cap M/M}\otimes \left.\mathscr{L}^{\vee}\right|_{\mr{L\cap M}}\otimes \left.\mathscr{M}\right|_{\mr{L\cap M}}}$.
 \2
\begin{theorem}\label{thm02}
 Let $\mr{X}/\C$ be a holomorphic symplectic variety. Suppose that $i : \rm L \xhookrightarrow{} X$, $j : \mr{M \xhookrightarrow{} X}$ are smooth Lagrangians with a compact Kähler intersection $\mr{L}\cap \mr{M}$ of codimension $c$ in $\mr{L}$. Assume that $\mr{K}_{\mr{L}}^{1/2}$ and $\mr{K}_\mr{M}^{1/2}$ exist and let $\mathscr{K}_\mr{or} \coloneqq \left(\mr{ \left.K_{L}^{1/2}\right|_{\mr{L}\cap\mr{M}}\otimes \left.K_{M}^{1/2}\right|_{\mr{L}\cap\mr{M}}}\right)^{\vee} \otimes \mr{K_{\mr{L}\cap\mr{M}}}$. Then the $\mr{Ext}$ local-to-global spectral sequence  \begin{equation*}\mathrm{E}_{2}^{p,q}=\mr{H}^{p}(\mr{L}\cap \mr{M},\Omega^{q-c}_{\mr{L}\cap \mr{M}} \otimes \mathscr{K}_{\mr{or}})\Rightarrow \mathrm{Ext}^{p+q}\big(i_{*}\mr{K_{L}^{1/2}},j_{*}\mr{K_{M}^{1/2}}\big)\end{equation*} degenerates on the second page. In particular, $$\mr{Ext}^{k}\big(i_{*}\mr{K_{L}^{1/2}},j_{*}\mr{K_{M}^{1/2}}\big) = \oplus_{p,q} \mr{H}^{p}(\mr{L}\cap \mr{M},\Omega^{q-c}_{\mr{L}\cap \mr{M}} \otimes \mathscr{K}_{\mr{or}}) = \mr{H}^{k-c}(\mr{L}\cap\mr{M},\mathfrak{K}_\mr{or}),$$ where $\mathfrak{K}_\mr{or}$ is the local system corresponding to the $2$-torsion line bundle $\mathscr{K}_\mr{or}$.
\end{theorem}
\2
\begin{remark}\label{page2}
 As above, we remark that the proof shows it is enough to assume the intersection $\mr{L}\cap \mr{M}$ smooth, compact such that its Hodge-to-de Rham spectral sequence for the local system $\mathfrak{K}_\mr{or}$ degenerates on $\mr{E}_1$.
\end{remark}
$\hspace*{3mm}$Notice that $\mr{Ext}\big(i_{*}\mr{K_{L}^{1/2}},j_{*}\mr{K_{M}^{1/2}}\big)$ is a graded module over $\mr{Ext}\big(i_{*}\mr{K_{L}^{1/2}},i_{*}\mr{K_{L}^{1/2}}\big)$. A theorem of Deligne (see \cite{MR972343}, \cite{MR1179076}) asserts that $\Omega^{*}(\mr{L}\cap\mr{M},\mathfrak{K}_\mr{or})$ is a formal dg module over $\Omega^{*}(\mr{L}\cap\mr{M},\C)$. Let $\Omega_{\bar{\partial}}^{*}$ be the complex of $\bar{\partial}$-closed forms with differential $\partial$. The diagram
\[
\begin{tikzcd}[column sep=small,row sep=small]
\Omega^{*}(\mr{L},\C) \arrow[d] & \Omega_{\bar{\partial}}^{*}(\mr{L},\C)\arrow[l] \arrow[d] \arrow[r] 
 &  \mr{H}(\mr{L}/\C) \arrow[d] \\
\Omega^{*}(\mr{L}\cap\mr{M},\C) & \Omega_{\bar{\partial}}^{*}(\mr{L}\cap\mr{M},\C) \arrow[l]  \arrow[r]  
  & \mr{H}(\mr{L}\cap\mr{M}/\C)
\end{tikzcd}
\]
shows $\Omega^{*}(\mr{L}\cap\mr{M},\mathfrak{K}_\mr{or})$ is also formal over $\Omega^{*}(\mr{L},\C)$. Hence $\mr{H}(\mr{L}\cap\mr{M},\mathfrak{K}_\mr{or})$ is formal as an $\Aa$-module over $\mr{H}(\mr{L}/\C)$. We prove the same goes for the $\mr{Ext}$ modules:
\2
\begin{theorem}\label{thm03}
 Let $\mr{X}/\C$ be holomorphic symplectic. Suppose that $i : \rm L \xhookrightarrow{} X$, $j : \mr{M \xhookrightarrow{} X}$ are compact Kähler Lagrangians with a smooth intersection. Assume that their canonical bundles admit square roots. Then $\mr{RHom}\big(i_{*}\mr{K_{L}^{1/2}}, j_{*}\mr{K_{M}^{1/2}}\big)$ is a formal differential graded module over the (formal) differential graded algebra $\mr{RHom}\big(i_{*}\mr{K_{L}^{1/2}},i_{*}\mr{K_{L}^{1/2}}\big)$. Moreover, we have a quasi-isomorphism of pairs $$\big(\mr{RHom}\big(i_{*}\mr{K_{L}^{1/2}},i_{*}\mr{K_{L}^{1/2}}\big),\mr{RHom}\big(i_{*}\mr{K_{L}^{1/2}}, j_{*}\mr{K_{M}^{1/2}}\big)\big)\simeq \big(\mr{H}(\mr{L}/\C),\mr{H}^{*-c}(\mr{L}\cap\mr{M},\mathfrak{K}_\mr{or})\big),$$ where $c$ is the codimension of $\mr{L}\cap \mr{M}$ in $\mr{L}$.
\end{theorem}
\2
\begin{remark}A few observations regarding \cref{thm03}:
 \begin{enumerate}
  \item The proof shows that it is enough to assume $\mr{L}$ compact K\"{a}hler and $\mr{M}$ smooth such that $\mr{L}\cap\mr{M}$ is smooth.
  \item A variant for the dg module structure over $\mr{RHom}\big(j_*\mr{K}_\mr{M}^{1/2},j_*\mr{K}_\mr{M}^{1/2}\big)$ can be formulated, reversing the assumptions on $\mr{L}$ and $\mr{M}$ in the part $1.$ of the remark.
 \end{enumerate}
\end{remark}

\paragraph{Method of proof.}
The proof of \cref{thm01} involves two main ingredients. We first observe that the dg algebra $\mr{RHom}\big(i_{*}\mr{K_{L}^{1/2}}, i_{*}\mr{K_{L}^{1/2}}\big)$ can be deformed over $\C\brak$ to the de Rham complex - this involves results of Schapira et al. on deformation quantisation \cite{MR2331247}, \cite{10.2307/40068123}, hence the square root of the canonical bundle. The deformation is the de Rham complex $\Omega^{*}(\mr{L},\C)$ over the generic point, therefore it is generically formal by K\"{a}hler formality of \cite{Deligne1975}. Applying semicontinuity yields degeneration of the spectral sequence, proving the degeneration for $\mr{K}_\mr{L}^{1/2}$. The formality requires a second ingredient - it is a theorem of Kaledin that generically formal dg algebras are formal (under certain conditions which our family satisfies).\\
$\hspace*{3mm}$The proofs of \cref{thm02} and \cref{thm03} follow the same lines, however, to prove the formality, we need \cref{modformality} - the analogue to the result of Kaledin for $\Aa$-modules.\\
$\hspace*{3mm}$Kaledin in \cite{MR2372207} treats formality of dg algebras as triviality of the normal cone deformation. Our approach to formality of $\Aa$-modules is also motivated by the deformation to the normal cone as we now explain, although we phrase our results without reference to deformation theory, opting to use the language of $\mr{A}_n$-algebras and modules instead.\\ 
$\hspace*{3mm}$Let $\mr{A}$ be a graded algebra over $\mr{R}$. Consider a minimal $\Aa$-module $\mr{M}$ over $\mr{A}$. Assume $\mr{A,M}$ projective over $\mr{R}$. Let $\tilde{\mr A} = \mr{A}[h]$ be the trivial deformation of $\mr{A}$ over $\mr{R}[h]$. Consider the graded $\mr{R}[h]$-module $\mr{M}[h]$. Then $(m_2^{\mr M},m_{3}^{\mr M}h,m_4^{\mr M}h^2,\cdots)$ turns $\mr{M}[h]$ into an $\Aa$-module over $\mr{A}[h]$ since the defining relations $\eqref{defeq3}$ for $\Aa$-modules are homogeneous. Letting $\tilde{\mr{M}}$ be the so defined $\Aa$-module, observe that the general fibre is $\mr M$. We write $\mr{M}(2)$ for the minimal (formal) $\Aa$-module $(\mr{M},(m^{\mr{M}}_2,0,0,\cdots)$, so the central fibre of $\tilde{\mr M}$ is $\tilde{\mr{M}}/h = \mr{M}(2)$. In light of this, the following definition seems natural.
\2
\begin{definition}
 The $\Aa$-module $\tilde{\mr M}$ is the deformation of $\mr M$ to the normal cone.
\end{definition}
\1
$\hspace*{3mm}$One should think of $\tilde{\mr{M}}$ as an $\Aa$-deformation of $\mr{M}(2)$ to $\mr{M}$. Notice that the formality of the $\Aa$-module $\tilde{\mr{M}}$ is the same as triviality of $\tilde{\mr{M}}$ as a deformation, i.e. in either case we are asking for a quasi-isomorphism  $\tilde{\mr{M}} \simeq \mr{M}(2)[h]$. Reducing modulo $h-1$, we see that this implies that $\mr{M}$ is formal. The converse is also true. Hence formality of $\mr{M}$ is equivalent to triviality of the deformation $\tilde{\mr{M}}$, so we can use obstruction theory and cohomology, standard deformation theory tools, to find formality criteria for $\Aa$-modules.
\paragraph{\textbf{Context.}}Solomon and Verbitsky \cite{solomon_verbitsky_2019} study the Fukaya category of $\mr{I}$-holomorphic graded spin Lagrangians in a hyperkähler variety $(\mr{X},\mr{I,J,K},g)$, equipped with the symplectic form $\omega_\mr{J}=g(\mr{J}\cdot,\cdot)$. Recall that spin is needed to set up Floer theory and is equivalent to choosing square roots of the canonical bundles in the complex case. When $\mr{L}\cap\mr{M}$ is smooth, they show that the Floer coboundary operator $\mu_1$ on $\mr{CF}(\mr{L,M})$ coincides with the de Rham differential, hence $\mr{HF}(\mr{L},\mr{M})$ is the de Rham cohomology of $\mr{L}\cap \mr{M}$, up to Maslov index shifts and tensoring by $\Lambda$. Moreover, for $\mr{CF}(\mr{L,L})$, $\mu_2$ is the wedge product of differential forms up to sign, while $\mu_k=0$ for $k \ge 3$. Thus, in the compact case, they recover a formality result of Ivan Smith: the proof of Kähler formality as in \cite{Deligne1975} shows that the Floer $\Aa$-algebra $\mr{CF}(\mr{L},\mr{L})$ is formal. We note that their results might be thought of as mirror to ours.\\
$\hspace*{3mm}$More generally, Solomon and Verbitsky consider a collection of $\mr{I}$-holomorphic graded spin compact Lagrangians $\mathcal{L}$ and define a Fukaya $\Aa$-category $\widehat{\mathcal{A}}_\mathcal{L}$. In light of the formality of $\mr{CF}(\mr{L,L})$, it's natural to ask if $\widehat{\mathcal{A}}_\mathcal{L}$ is a formal $\Aa$-category. It is known to be intrinsically formal in the case of the $\mr{A}_k$-Milnor fibre by \cite{MR1831820} and \cite{MR3486414} shows formality of $\widehat{\mathcal{A}}_\mathcal{L}$ in the case of a transverse nilpotent slice of the adjoint quotient $\mathfrak{sl}_2k(\C) \to \C^{2k-1}$ and a distinguished (finite) collection of Lagrangians defined in \cite{MR2254624}. We answer a similar question in the setting of virtual de Rham cohomology - see \cref{bfconj} below.\\
$\hspace*{3mm}$The work of Kapustin \cite{2004IJGMM..01...49K} gives physics motivation that the Fukaya category of a hyperkähler variety $(\mr{X},\mr{I,J,K})$ with symplectic form $\omega_\mr{J}$ should be equivalent to the category of DQ modules on the holomorphic symplectic manifold $(\mr{X},\mr{I},\sigma_\mr{I} = \omega_\mr{J}+i\omega_\mr{K})$. Following this observation, Solomon and Verbitsky speculate that the category $\widehat{\mathcal{A}}_\mathcal{L}$ should be equivalent, in some sense, to the full subcategory $\mathcal{D}_\mathcal{L}$ of the category of DQ modules on $\mr{X}$ whose objects are simple holonomic modules $\mathscr{D}_\mr{L}$ along $\mr{L} \in \mathcal{L}$.\\
$\hspace*{3mm}$There is also \cite{MR2641169} where Behrend and Fantechi sketch the construction of a differential graded category $\mathcal{C}$ of Lagrangians in a holomorphic symplectic variety $\mr{X}$. This construction depends on their constructible virtual de Rham complex $\big(\mathscr{E} \coloneqq \ext\big(i_*\mr{K}^{1/2}_\mr{L},j_*\mr{K}^{1/2}_\mr{M}\big),\mr{d}\big)$. Locally $\mr{X}$ is the cotangent bundle of $\mr{M}$ and $\mr{L}$ is given by the graph of $\mr{d}f$ for some $f \in \Gamma(\mr{M},\mathscr{O}_{\mr{M}})$, hence $\mr{L}\cap \mr{M}$ is the critical locus $\mr{crit}f$. Then the sheaves $\mathscr{E}$ are the cohomology sheaves of $(\Omega_\mr{M},\mr{d}f\wedge)$. Since $\mr{d}f\wedge$ anticommutes with the de Rham differential $\mr{d}$, we see that $\mr{d}$ descends to a differential on $\mathscr{E}$, still denoted $\mr{d}$. The main result of \cite{MR2641169} is that these locally defined differentials glue. The cohomology of the morphism complexes $\mr{Hom}_{\mathcal{C}}(\mr{L},\mr{M})$ is given by the virtual de Rham cohomology $\mr{R}\Gamma(\mathscr{E},\mr{d})$ of the intersection $\mr{L}\cap \mr{M}$. There is a spectral sequence computing $\mr{R}\Gamma(\mathscr{E},\mr{d})$: \begin{equation}\label{hypercohss}\mr{E}_1^{p,q} = \mr{H}^{q}(\mr{X},\mathscr{E}^{p}) \Rightarrow \mr{R}^{p+q}\Gamma(\mathscr{E},\mr{d}).\end{equation} When $\mr{L}\cap \mr{M}$ is compact Kähler, it degenerates by Hodge theory, hence our results prove a corrected version of \cite[Conjecture~5.8]{MR2641169}:
\2
\begin{proposition}\label{bfconj}
 Let $\mr{X}/\C$ be a holomorphic symplectic variety. Suppose that $i : \rm L \xhookrightarrow{} X$ and $j : \mr{M \xhookrightarrow{} X}$ are smooth Lagrangians such that $\mr{K}_{\mr{L}}^{1/2}$ and $\mr{K}_\mr{M}^{1/2}$ exist and $\mr{L}\cap \mr{M}$ is smooth compact. Assuming that the spectral sequence \eqref{hypercohss} degenerates on $\mr{E}_1$, e.g. if $\mr{L}\cap\mr{M}$ is furthermore K\"{a}hler, we have $$\mr{R}^{k}\Gamma(\mathscr{E},\mr{d}) = \mr{Ext}^{k}\big(i_{*}\mr{K}_\mr{L}^{1/2},j_*\mr{K}_\mr{M}^{1/2}\big).$$
\end{proposition}
\begin{remark}
This conjecture is the analogue of the formality of $\widehat{\mathcal{A}}_\mathcal{L}$ in the virtual setting.
\end{remark}

\paragraph{Plan of paper.}
In \cref{sec1} we start by reviewing general results on $\Aa$-algebras and conclude with the formality theorems of Kaledin \cite{MR2372207}. In the next paragraph we define $\Aa$-modules and recall standard results such as Kadeishvili's theorem on minimal models. Then we have a paragraph on (bi-graded) Hochschild cohomology of modules over graded algebras. It culminates in some results on base-change for Hochschild cohomology, mirroring statements in \cite{MR2578584}. After that, we develop obstruction theory for extending $\mr{A}_n$-modules to $\mr{A}_{n+1}$-modules as well as $\mr{A}_n$-morphisms to $\mr{A}_{n+1}$-morphisms between modules - this is motivated by similar ideas in \cite{MR1480721} for $\mr{L}_\infty$-algebras and \cite{lefvrehasegawa2003sur} dealing with $\Aa$-algebras. The last paragraph contains our results on formality of $\Aa$-modules.\\
$\hspace*{3mm}$In the next section, \cref{sec2}, we recall results of \cite{bbdjs} on perverse sheaves on $\mr{d}$-critical loci - a structure that exists on the intersection of two Lagrangians. Then we have a reminder on deformation quantisation modules, following \cite{MR2331247}, \cite{MR3012169}. In particular, we compare the compatibility of \cite{MR2331247} to the results by Ginzburg et al. in \cite{MR3437831}. We conclude by relating perverse sheaves on Lagrangian intersections and simple holonomic DQ modules.\\
$\hspace*{3mm}$In \cref{sec3} we start by applying deformation quantisation to prove the degeneration of the local-to-global $\mr{Ext}$ spectral sequence, for square roots of the canonical bundle, in both cases of single Lagrangian and a pair of cleanly intersecting Lagrangians. We conclude with our main results on formality.

\paragraph{Conventions.}
We work with $\Aa$-algebras and modules over a (unital) ring $\mr{R}$ and we shall assume that $\mr{R}$ is a commutative algebra over a field $k$ of characteristic $0$.

\paragraph{Acknowledgements.}
I would like to thank my supervisor Richard Thomas for suggesting the formality problems discussed here, countless helpful discussions, suggestions and corrections. Thanks to François Petit and Pierre Schapira for sharing their expertise in DQ modules and to Jake Solomon for useful comments. This work has been supported by EPSRC [EP/R513052/1], President's PhD Scholarship, Imperial College London. 

\section{\texorpdfstring{$\Aa$}--algebras}\label{sec1}

$\hspace*{3mm}$A good general reference for $\Aa$-algebras is \cite{lefvrehasegawa2003sur}. We recall $\Aa$-algebras over a ring $\mr{R}$, their bar construction, Kadeishvili's theorem on minimal models and formality results of Kaledin-Lunts. Then we review definitions and standard results on $\Aa$-modules. The final paragraphs of the section contain Hochschild cohomology, obstruction theory and our results on formality of $\Aa$-modules.
\paragraph{\texorpdfstring{$\Aa$}--algebras.}

\begin{definition}
 Let $n \in \mathbf{N} \cup \{\infty \}$. An $\mr A_n$-algebra is a graded $\mr{R}$-module $\mr{A}$ equipped with a family of $\mr{R}$-linear morphisms $$m_{i}: \mr{A}^{\otimes i} \to \mr{A}$$ of degree $2-i$ for $1 \le i \le n$ such that for all $m \le n$ we have
 \begin{equation}\label{defeq1} \tag{$*_m$}
 \sum_{j+k+l=m} (-1)^{jk+l}m_{j+1+l}(\mr{id}^{\otimes j}\otimes m_{k}\otimes \mr{id}^{\otimes l})=0.
 \end{equation}
\end{definition}
\begin{remark}
\begin{enumerate}
 \item If $\mr{m}_{i}=0$ for all $i\not=2$, then $\mr{A}$ is a graded algebra.
 \item If $m_i=0$ for all $i\not=1,2$, $\mr{A}$ is a differential graded algebra.
\end{enumerate}
\end{remark}
\2
\begin{example}
 Let $\mr{A}$ be an $\Aa$-algebra. Then the first relation is $$(*_1) \quad m_1m_1 =0,$$ i.e. $m_1$ is a differential. The second relation is $$(*_2) \quad m_1m_2 = m_2(m_1\otimes \mr{id} + \mr{id} \otimes m_1),$$ meaning that $m_1$ is a derivation for the multiplication $m_2$. The third equation shows $m_2$ is associative up to the homotopy $m_3$: $$(*_3) \quad m_2(m_2\otimes \mr{id} - \mr{id}\otimes m_2) = m_1m_3+m_3(m_1\otimes \mr{id}^{\otimes 2} + \mr{id}\otimes m_1 \otimes \mr{id} + \mr{id}^{\otimes 2}\otimes m_1).$$ In particular, any minimal $\mr{A}_n$-algebra is associative. The cohomology of an $\mr{A}_n$-algebra is a graded associative algebra.
\end{example}
\2
\begin{definition}
 Let $n \in \mathbf{N} \cup \{\infty \}$. Let $\mr{A}, \mr{B}$ be $\mr A_n$-algebras over $\mr{R}$. An $\mr{A}_n$-morphism $f:\mr{A} \to \mr{B}$ between $\mr A_n$-algebras is a family of $\mr{R}$-linear morphisms $$f_{i} : \mr{A}^{\otimes i} \to \mr{B}$$ of degree $1-i$ for $1 \le i \le n$ such that for all $m\le n$ we have \begin{equation}\label{defeq2}\tag{$**_m$}
 \sum_{j+k+l=m} (-1)^{jk+l}f_{j+1+l}(\mr{id}^{\otimes j}\otimes m_{k}\otimes \mr{id}^{\otimes l}) =  \sum_{i_{1}+\cdots +i_{r}=m}(-1)^sm_{r}(f_{i_1}\otimes \cdots \otimes f_{i_r}),                                                                                                                                                                                                                                                                                                                                                      \end{equation}
 where we set $$s = \sum_{2\le u\le r}\big((1-i_u)\sum_{1\le v \le u} i_v\big).$$
The composition of $f : \mr{A} \to \mr{B}$ and $g : \mr{B} \to \mr{C}$ is defined by $$(g \circ f)_n = \sum_r \sum_{i_1 +\cdots+i_r=n} (-1)^sg_{r}(f_{i_{1}}\otimes \cdots \otimes f_{i_{r}}).$$ 
\end{definition}
\begin{example}
 Let $f : \mr{A} \to \mr{B}$ be an $\Aa$-morphism. Then $$(**_1) \quad f_1m_1=m_1f_1,$$ that is, $f_1$ is a morphism of complexes, i.e. $\mr{A}_1$-morphisms are just morphisms of complexes. The second relation is $$(**_2) \quad f_1m_2 = m_2(f_1\otimes f_1) + m_1f_2+f_2(m_1\otimes \mr{id} +\mr{id}\otimes m_1),$$ measuring the compatibility of $f_1$ with the multiplications of $\mr{A}$ and $\mr{B}$. 
\end{example}
\2
\begin{remark}
We denote the category of $\mr A_n$-algebras and $\mr A_n$-morphisms by $\mr{Alg}_{n}$ and the category of $\Aa$-algebras by $\mr{Alg}_{\infty}$. The category of differential graded algebras is a non-full subcategory of $\mr{Alg}_{\infty}$.
\end{remark}
\2
\begin{definition}
 Let $n$ be a positive integer or $\infty$. 
 \begin{enumerate}
  \item  A morphism $f = (f_1,f_2,\cdots,f_n) : \mr{A} \to \mr{B}$ of $\mr A_n$-algebras is a quasi-isomorphism if $f_1$ is a quasi-isomorphism of the underlying complexes.
 \item $\mr{A}$ and $\mr{B}$ are said to be quasi-isomorphic if there exist $\mr{A}_n$-algebras $\mr{C}_1,\cdots \mr{C}_m$ and quasi-isomorphisms $\mr{A} \leftarrow \mr{C}_1 \rightarrow \cdots \leftarrow \mr{C}_m \rightarrow \mr{B}$.
 \end{enumerate}
\end{definition}
\2
\begin{definition}Let $n$ be a positive integer or $\infty$. 
\begin{enumerate}
 \item An $\mr{A}_n$-algebra $\mr{A}$ is called minimal if $m_1=0$.
 \item A minimal model for $\mr{A}$ is a minimal $\mr{A}_n$-algebra $\mr B$ together with a quasi-isomorphism $\mr{B} \to \mr{A}$.
 \end{enumerate}
 
\end{definition}
\2
\begin{theorem}(Kadeishvili \cite{MR580645})
 Let $\mr{A}$ be an $\Aa$-algebra over $\mr{R}$ such that $\mr{HA}$ is a projective $\mr{R}$-module. For any choice of a quasi-isomorphism $f_1 : \mr{HA} \to \mr{A}$ of complexes of $\mr{R}$-modules, there exists a minimal $\Aa$-structure on $\mr{HA}$, with $m^{\mr{HA}}_2$ being induced by $m_2$, and an $\Aa$-quasi-isomorphism $f : \mr{HA} \to \mr{A}$ lifting $f_1$.
\end{theorem}
\2
\begin{definition}
 Let $\mr{A}$ be an $\Aa$-algebra.
 \begin{enumerate}
  \item $\mr{A}$ is called $\mr{A}_n$-formal if it is $\mr{A}_n$-quasi-isomorphic to the $\mr{A}_n$-algebra $(\mr{HA},m^{\mr{HA}})$, where $m_2^{\mr{HA}}$ induced by $m_2$ and $m_i^{\mr{HA}}=0$ for $i \not= 2$.
  \item $\mr{A}$ is called formal if it is $\Aa$-quasi-isomorphic to the graded associative algebra $\mr{HA}$, viewed as an $\Aa$-algebra.
 \end{enumerate}
\end{definition}
\2
$\hspace*{3mm}$We have the following two important results due to Kaledin and Lunts.
\2
\begin{theorem}(Lunts \cite{MR2578584})
 Let $\mr{A}$ be a minimal $\Aa$-algebra over $\mr{R}$ which is projective as an $\mr{R}$-module. Then $\mr{A}$ is formal iff it is $\mr{A}_n$-formal for all $n$.
\end{theorem}
\1
Furthermore, Kaledin, in \cite{MR2372207}, shows that $\mr{A}_n$-formality is measured by a cohomology class, called the Kaledin class, which gives the next result.
\2
\begin{theorem}\label{formalityalg}(Kaledin-Lunts \cite{MR2372207},\cite{MR2578584})
 Let $\mr{R}$ be an integral domain with field of fractions $k(\eta)$. Consider a minimal $\Aa$-algebra $\mr{A}$ over $\mr{R}$ which is a finite projective $\mr{R}$-module. Assume that the Hochschild cohomology group with compact supports $\mr{HH}^{2}_{\mr{c}}(\mr{A}(2))$ is torsion-free. If $\mr{A}_\eta = k(\eta) \otimes_{\mr{R}} \mr{A}$ is formal, then $\mr{A}$ is formal. In particular, $\mr{A}_{\mathfrak{p}}$ is formal for all $\mathfrak{p} \in \mr{SpecR}$.
\end{theorem}

\paragraph{The bar construction.}

Let $\mr{A}$ be a graded $\mr{R}$-module endowed with morphisms $$m_{i}: \mr{A}^{\otimes i} \to \mr{A}.$$ For $i \ge 1$ we have a bijection 
\begin{align*}
\mr{Hom}(\mr{A}^{\otimes i},\mr{A}) &\to \mr{Hom}((\mr{A}[1])^{\otimes i},\mr{A}[1])\\
 m_i &\mapsto \mr{d}_i=(-1)^{i-1+\mr{deg}m_i}s\circ m_i\circ (s^{-1})^{\otimes i},                                                                                                                                                
\end{align*}
 where $s:\mr{A} \to \mr{A}[1]$ is the canonical degree $-1$ morphism. Remark that in our case $m_i$ are of degree $2-i$, so the corresponding $\mr{d}_i$ have degree $1$. The morphisms $\mr{d}_i$ define a unique morphism $$\overline{\mr{T}}(\mr{A}[1]) \to \mr{A}[1],$$ which by the universal property of the reduced tensor coalgebra corresponds to a unique degree $1$ coderivation $$\mr{d} : \overline{\mr{T}}(\mr{A}[1]) \to \overline{\mr{T}}(\mr{A}[1]).$$
 
\begin{lemma}
 The morphisms $m_i$ define an $\Aa$-algebra structure on $\mr{A}$ iff $\mr{d}$ is a differential, i.e. $\mr{d}^2=0$.
\end{lemma}
\2
\begin{definition}
 The bar construction of an $\Aa$-algebra $\mr{A}$ is the differential graded coalgebra $\mathcal{B}(\mr{A})\coloneqq(\overline{\mr{T}}(\mr{A}[1]),\mr d)$.
\end{definition}
\2
$\hspace*{3mm}$Let $\mr{A}, \mr{B}$ be graded objects. For $i\ge 1$ we have a bijection
\begin{align*}
\mr{Hom}(\mr{A}^{\otimes i},\mr{B}) &\to \mr{Hom}((\mr{A}[1])^{\otimes i},\mr{B}[1])\\
 f_i &\mapsto \mr{F}_i=(-1)^{i-1+\mr{deg}f_i}s_{\mr{B}}\circ f_i\circ (s_{\mr{A}}^{-1})^{\otimes i}.                                                                                                                                                
\end{align*}
If $f_i$ are of degree $1-i$, the maps $\mr{F}_i$ define a degree $0$ morphism of coalgebras $$\mr{F}:\mathcal{B}(\mr{A}) \to \mathcal{B}(\mr{B}).$$
\begin{lemma}
 Let $\mr{A}, \mr{B}$ be $\Aa$-algebras and suppose given $f_i \in \mr{Hom}(\mr{A}^{\otimes i},\mr{B})$ of degree $1-i$. The morphisms $f_i$ define an $\Aa$ morphism iff $F$ is compatible with the differentials, i.e. we have a bijection
 $$\mr{Hom}_{\mr{Alg}_{\infty}}(\mr{A},\mr{B}) \xrightarrow{\sim} \mr{Hom}(\mathcal{B}(\mr{A}),\mathcal{B}(\mr{B})).$$
\end{lemma}

\paragraph{\texorpdfstring{$\Aa$}--modules.} 

\begin{definition}
 Let $n \in \mathbf{N} \cup \{\infty \}$ and let $\mr{A}$ be an $\mr A_n$-algebra over $\mr{R}$. An $\mr A_n$-module over $\mr{A}$ is a graded $\mr{R}$-module $\mr{M}$ together with a family of morphisms $$m^{\mr{M}}_i: \mr{M}\otimes \mr{A}^{\otimes i-1} \to \mr{M}$$ of degree $2-i$ for all $1 \le i \le n$ such that for all $1\le m  \le n$ 
 \begin{equation}\label{defeq3} \tag{$*'_m$}
  \sum_{j+k+l=m,j\ge 1} (-1)^{jk+l}m^{\mr{M}}_{j+1+l}(\mr{id}^{\otimes j}\otimes m_{k}\otimes \mr{id}^{\otimes l})+ \sum_{k+l=m} (-1)^{l}m^{\mr{M}}_{1+l}(m^{\mr{M}}_{k}\otimes \mr{id}^{\otimes l})=0.
 \end{equation}
\end{definition}
\begin{definition}
 Let $n \in \mathbf{N} \cup \{\infty \}$ and let $\mr{A}$ be an $\mr A_n$-algebra and suppose $\mr{M}, \mr{N}$ are $\mr A_n$-modules over $\mr{A}$. A morphism of $\mr A_n$-modules is a family of $\mr{R}$-linear morphisms $$f_i : \mr{M}\otimes \mr{A}^{\otimes i-1} \to \mr{N}$$ of degree $1-i$ for $1 \le i \le n$ such that for all $1 \le m \le n$ 
 \begin{equation*}\label{defeq4}\tag{$**'_m$}                                                                                                                                                   \sum_{j+k+l=m, k\ge1} (-1)^{jk+l}f_{j+1+l}(\mr{id}^{\otimes j}\otimes m_{k}\otimes \mr{id}^{\otimes l}) = \sum_{r+s=m, r\ge 1, s\ge 0}m_{s+1}(f_{r}\otimes \mr{id}^{\otimes s}).                                                                                                                                                                                       \end{equation*}
The composition of $f:\mr{L} \to \mr{M}$ and $g:\mr{M}\to \mr{N}$ is defined by $$(g \circ f)_n = \sum_{k+l=n} g_{l+1}(f_{k}\otimes 1^{\otimes l}).$$
\end{definition}
\begin{remark}
 Let $\mr{A}$ be an $\Aa$-algebra and $\mr{M}$ be an $\Aa$-module over $\mr{A}$. Then
 \begin{enumerate}
  \item $(\mr{M},m^{\mr{M}}_1)$ is a complex;
  \item If $f :\mr{M} \to \mr{N}$ is a morphism of $\Aa$-modules, $f_1$ is a morphism of complexes $$f_1:(\mr{M},m^{\mr{M}}_1)\to (\mr{N},m^{\mr{N}}_1).$$
 \end{enumerate}
\end{remark}
\2
\begin{example}
 If $\mr{A}$ is an $\Aa$-algebra, then the morphisms $m_i :\mr{A}^{\otimes i} \to \mr{A}$ define an $\Aa$-module structure on $\mr{A}$ over $\mr{A}$.
\end{example}
\2
\begin{remark}
 If $\mr{A}$ is a differential graded algebra regarded as $\Aa$-algebra, then any differential graded module over $\mr{A}$ is canonically $\Aa$-modules and the category of differential graded modules over $\mr{A}$ is a non-full subcategory of the category of $\Aa$-modules over $\mr{A}$.
\end{remark}
\2
\begin{definition}
Let $n$ be a positive integer or $\infty$. Let $\mr{A}$ be an $\mr{A}_n$-algebra and suppose $\mr{M}$ and $\mr{N}$ are $\mr{A}_n$-modules over $\mr{A}$. 
\begin{enumerate}
\item An $\mr{A}_n$-morphism $f = (f_1,f_2,\cdots,f_n):\mr{M} \to \mr{N}$ is a quasi-isomorphism if $f_1$ is a quasi-isomorphism of complexes.
\item $\mr{M}$ and $\mr{N}$ are said to be quasi-isomorphic if there exist $\mr{A}_n$-modules $\mr{M}_1,\cdots \mr{M}_m$ and quasi-isomorphisms $\mr{M} \leftarrow \mr{M}_1 \rightarrow \cdots \leftarrow \mr{M}_m \rightarrow \mr{N}$.
 \end{enumerate}
\end{definition}
\2
\begin{definition}
Let $n$ be a positive integer or $\infty$. Let $\mr{A}$ be an $\mr{A}_n$-algebra and consider an $\mr{A}_n$-module $\mr{M}$ over $\mr{A}$.
\begin{enumerate}
 \item $\mr{M}$ is called minimal if $m^{\mr{M}}_1=0$.
 \item A minimal model for $\mr{M}$ is a pair $(\mr{A'},\mr{M'})$, consisting of a minimal $\mr{A}_n$-algebra $\mr{A'}$ and a minimal $\mr{A}_n$-module $\mr{M'}$ over it, together with quasi-isomorphisms $f:\mr{A'} \to \mr{A}$ and $g : \mr{M'} \to f^{*}\mr{M}$, where $f^{*}\mr{M}$ is the restriction of $\mr{M}$ along $f$.
\end{enumerate}
\end{definition}
\2
\begin{remark}
 We shall say that $(f,g)$ is a morphism of pairs $(\mr{A'},\mr{M'}) \to (\mr{A},\mr{M})$.
\end{remark}
\2
\begin{theorem}(Kadeishvili, \cite{MR580645})
 Let $\mr{A}$ be an $\Aa$-algebra and consider an $\Aa$-module $\mr{M}$ over $\mr{A}$. Assume that $\mr{H}\mr{A}$ and $\mr{H}\mr{M}$ are projective $\mr R$-modules. Then, for any choice of quasi-isomorphisms $$f_1:\mr{HA}\to \mr{A}, g_1 :\mr{HM} \to \mr{M}$$ of complexes of $\mr{R}$-modules, inducing the identity in cohomology, there exists a minimal $\Aa$-module structure on  $\mr{HM}$ over the $\Aa$-algebra $\mr{HA}$, with $m_2^{\mr{HM}}$ induced by $m_2^{\mr{M}}$, such that there exists a quasi-isomorphism of pairs $$(f,g):(\mr{HA},\mr{HM}) \to (\mr{A},\mr{M}),$$ lifting $(f_1,g_1)$, i.e. a minimal model for $\mr{M}$. It is unique up to $\Aa$-isomorphism.
\end{theorem}
%
%

\paragraph{The bar construction for \texorpdfstring{$\Aa$}--modules.}

Let $\mr{A}$ and $\mr{M}$ be graded $\mr{R}$-modules. For $i \ge 1$ we have a bijection 
\begin{align*}
\mr{Hom}(\mr{M}\otimes \mr{A}^{\otimes i-1},\mr{M}) &\to \mr{Hom}(\mr{M}[1]\otimes (\mr{A}[1])^{\otimes i-1},\mr{M}[1])\\
 m^{\mr{M}}_i &\mapsto \mr{d}^{\mr{M}}_i=(-1)^{i-1+\mr{deg}m^{\mr{M}}_i}s\circ m_i\circ (s^{-1})^{\otimes i}.
 \end{align*}
$\hspace*{3mm}$Let $\mr{A}$ be an $\Aa$-algebra and let $(\mathcal{B}(\mr{A}))^{+}$ be its coaugmented bar construction. Then the $\mr{d}^{\mr{M}}_i$ define a unique comodule coderivation $$\mr{d}^{\mr{M}} : \mr{M}[1]\otimes (\mathcal{B}(\mr{A}))^{+} \to \mr{M}[1]\otimes (\mathcal{B}(\mr{A}))^{+}.$$

\begin{lemma}
 The morphisms $m^{\mr{M}}_i$ define an $\Aa$-module structure of $\mr{M}$ over $\mr{A}$ iff the coderivation $\mr{d}^{\mr{M}}$ is a differential.
\end{lemma}
\2
$\hspace*{3mm}$Let $\mr{A},\mr{M},\mr{N}$ be graded $\mr{R}$-modules. For all $i \ge 1$ we have a bijection 
\begin{align*}
\mr{Hom}(\mr{M}\otimes \mr{A}^{\otimes i-1},\mr{N}) &\to \mr{Hom}(\mr{M}[1]\otimes(\mr{A}[1])^{\otimes i-1},\mr{N}[1])\\
 f_i &\mapsto \mr{F}_i=(-1)^{i-1+\mr{deg}f_i}s_{\mr{B}}\circ f_i\circ (s_{\mr{A}}^{-1})^{\otimes i}.                                                                                                                                                
\end{align*}
and, if $\mr{A}$ is an $\Aa$-algebra, the $\mr{F}_i$ induce a morphism of $(\mathcal{B}(\mr{A}))^{+}$-comodules $$\mr{F}:\mr{M}\otimes (\mathcal{B}(\mr{A}))^{+} \to \mr{N}\otimes (\mathcal{B}(\mr{A}))^{+}.$$

\begin{lemma}
 Let $\mr{A}$ be an $\Aa$-algebra and suppose given graded objects $\mr{M}, \mr{N}$. For all $i\ge1$ there is a bijection $$\mr{Hom}^{1-i+n}(\mr{M}\otimes \mr{A}^{\otimes i-1},\mr{N})\xrightarrow{\sim} \mr{Hom}^{n}_{\mathcal{B}(\mr{A})}(\mr{M}\otimes (\mathcal{B}(\mr{A}))^{+}, \mr{N}\otimes (\mathcal{B}(\mr{A}))^{+}).$$ Furthermore, if $\mr{M}$ and $\mr{N}$ are $\Aa$-modules, then we get an induced bijection between morphisms $\mr{M} \to \mr{N}$ of $\Aa$-modules and degree $0$ morphisms of differential graded $(\mathcal{B}(\mr{A}))^{+}$-comodules.
\end{lemma}

\paragraph{Differential graded pairs.}

\begin{definition}
 Let $\mr{A,B}$ be differential graded algebras over $\mr{R}$ and let $\mr{M}$ (resp. $\mr{N}$) be a differential graded module over $\mr{A}$ (resp. $\mr{B}$). 
 \begin{enumerate}
  \item  If $f:\mr{A}\to \mr{B}$ is a morphism of differential graded algebras and $g : \mr{M} \to f^{*}\mr{N}$ is a morphism of differential graded modules, where $f^{*}$ denotes restriction along $f$, we say that the pair $(f,g): (\mr{A,M}) \to (\mr{B,N})$ is a differential graded morphism of pairs.
  \item The pairs $(\mr{A,M})$ and $(\mr{B,N})$ are differential graded quasi-isomorphic, denoted $(\mr{A},\mr{M})\simeq(\mr{B},\mr{N})$, if there exist pairs $(\mr{A}_1,\mr{M}_1),\cdots,(\mr{A}_m,\mr{M}_m)$ and quasi-isomorphisms of pairs $$(\mr{A},\mr{M})\leftarrow (\mr{A}_1,\mr{M}_1) \rightarrow \cdots \leftarrow (\mr{A}_m,\mr{M}_m) \rightarrow (\mr{B},\mr{N}).$$
  \item If $\mr{A}$ is a formal differential graded algebra, we say that $\mr{M}$ is differential graded formal if the pairs $(\mr{A,M})$ and $(\mr{HA,HM)}$ are differential graded quasi-isomorphic.
 \end{enumerate}
 
\end{definition}

$\hspace*{3mm}$Let $\mr{A},\mr{B}$ be differential graded algebras over $\mr{R}$ such that $\mr{HA}$ and $\mr{HB}$ are projective $\mr{R}$-modules, equpped with their minimal $\Aa$-algebra structures. Given two differential graded modules $\mr{M}$ and $\mr{N}$ over $\mr{A}$ and $\mr{B}$, respectively, assume that $\mr{HM}$ and $\mr{HN}$ are projective over $\mr{R}$, so can be given minimal $\Aa$-module structures over $\mr{HA}$ and $\mr{HB}$. 
\2
\begin{proposition}\label{dgainf}
  The pairs $(\mr{A},\mr{M})$ and $(\mr{B},\mr{N})$ are differential graded quasi-isomorphic if and only if the pairs $(\mr{HA},\mr{HM})$ and $(\mr{HB,HN})$ are $\Aa$-quasi-isomorphic.
\end{proposition}

\paragraph{Hochschild cohomology.}

Let $\mr A$ be a graded algebra. We are going to define Hochschild cohomology for a graded module $\mr M$ over $\mr A$.\\
$\hspace*{3mm}$Let $\mr{C}^{p,q}(\mr{A,M}) = \mr{Hom}^{q}(\mr{M}\otimes \mr{A}^{\otimes p}, \mr{M})$. The module structure of $\mr M$ over $\mr A$ is a graded morphism of degree $0$, denoted by $m_{2}^{\mr{M}} : \mr{M}\otimes \mr{A} \to \mr{M}$. We can endow the modules $\mr{C}^{p,q}(\mr{A,M})$ with a differential, called the Hochschild differential:
\begin{align*} 
&\mr{d}  : \mr{C}^{p,q}(\mr{A,M}) \to \mr{C}^{p+1,q}(\mr{A,M})\\
f  \mapsto \sum (-1)^{l}f(\mr{id}^{\otimes j}\otimes & m_{2}^{\mr A} \otimes \mr{id}^{\otimes l}) - m_{2}^{\mr M}(f \otimes \mr{id}) + (-1)^pf(m_{2}^{\mr M}\otimes \mr{id}^{p}).
\end{align*}
A calculation shows that $\mr{d}^2=0$, so we indeed have a differential, the associated complex is called the Hochschild complex.
\2
\begin{definition}
 The Hochschild cohomology $\mr{HH}^{p,q}(\mr{A},\mr{M})$ of a graded module $\mr M$ over a graded algebra $\mr A$ is the $p^{\text{th}}$ cohomology of the Hochschild complex
 $(\mr{C}^{*,q}(\mr{A,M}),\mr{d})$.
\end{definition}
\2
\begin{example}Suppose $f : \mr{M}[\epsilon] \to \mr{M}[\epsilon]$ is an $\mr{A}[\epsilon]$-automorphism lifting the identity. Then, writing $$f = f^{0} + f^{1}\epsilon,$$ we have by assumption $f^{0} = \mr{id}_{\mr{M}}$, and $\mr{A}$-linearity implies $$f^{1}(ma) = f^{1}(m)a$$ which is to say that $f^{1}$ is a $(0,0)$-cocycle. 
\end{example}
\2
\begin{definition}
 An infinitesimal $\Aa$-deformation of an $\mr{A}$-module $\mr{M}$ is an $\Aa$-module structure on $\mr{M}[\epsilon]$ over $\mr{A}[\epsilon]$ extending the $\mr{A}$-module structure on $\mr{M}$.
\end{definition}
\2
\begin{example}The $(1,0)$-cocycles are precisely the $\mr{A}[\epsilon]$-module structures on $\mr{M}[\epsilon]$, i.e. they correspond to infinitesimal deformations. Indeed, if $$m : \mr{M}[\epsilon]\otimes \mr{A}[\epsilon] \to \mr{M}[\epsilon]$$ is the multiplication, we decompose it as $$m = m^{0}+m^{1}\epsilon,$$ where $m^{0}$ is the $\mr{A}$-module multiplication on $\mr{M}$. As $m$ defines a module structure, we get $$m^{1}(m,a)a'+m^{1}(ma,a') = m^{1}(m,aa'),$$ i.e. $m^{1}$ is a $(1,0)$-cocycle.
Notice that there is a canonical $\mr{A}[\epsilon]$-module structure on $\mr{M}[\epsilon]$.
\end{example}
\2
\begin{definition}
We call an infinitesimal $\Aa$-deformation of $\mr{M}$ trivial if it is quasi-isomorphic to this canonical one. 
\end{definition}
\2
\begin{example}We note that $(1,0)$-coboundaries correspond to trivial deformations by a similar calculation, hence infinitesimal deformations of $\mr{M}$ are classified by $\mr{HH}^{1,0}(\mr{A},\mr{M})$.\\
$\hspace*{3mm}$More generally, assume we are given an infinitesimal $\Aa$-deformation of $\mr{M}$. Write $$m = m^{0}+m^{1}\epsilon,$$ where $m^0$ is the $\mr{A}$-module structure on $\mr{M}$ and $m^1 = (m^1_2,m^1_3,\cdots)$. Notice the relations $\eqref{defeq3}$ are homogeneous in $\epsilon$, so as $\epsilon^2=0$ we see that each $m^1_i$ is a cocycle in the Hochschild complex of $\mr{M}$ and, conversely, any collection of cocycles satisfies $\eqref{defeq3}$ for all $m$. Similarly, we see coboundaries correspond to trivial $\Aa$-deformations, hence $$\prod_{n\ge 1} \mr{HH}^{n,1-n}(\mr{A,M})$$ classifies infinitesimal $\Aa$-deformations.
\end{example}
\2
\begin{remark}
 Remark that here we assume the deformation parameter is in degree $0$ with respect to the internal grading of our objects. If we consider $derived$ deformations, i.e. allow the deformation parameter to have a non-zero degree with respect to the internal grading, then we get different cohomology group, e.g. if it is in degree $1$, the cohomology group classifying derived infinitesimal deformations becomes $\mr{HH}^{1,-1}(\mr{A},\mr{M})$.
\end{remark}
\2
\begin{remark}
 Given two graded modules $\mr{M}$ and $\mr{N}$ over $\mr{A}$, let $\mr{C}^{p,q}(\mr{A,M,N}) = \mr{Hom}^{q}(\mr{M}\otimes \mr{A}^{\otimes p}, \mr{N})$. It carries a Hochschild differential and we define the Hochschild cohomology of the pair $(\mr{M,N})$, denoted $\mr{HH}^{p,q}(\mr{A,M,N})$, to be the cohomology of the resulting complex.
\end{remark}
\2
\begin{proposition}\label{hochbasechange}
 Let $\mr{A}$ be a graded algebra over $\mr{R}$ and let $\mr{M}$ be a graded module over $\mr{A}$. Assume that $\mr{A}$ and $\mr{M}$ are finite projective over $\mr{R}$. Suppose that $\mr{R} \to \mr{Q}$ is a morphism of commutative rings. Write $\mr{A}_\mr{Q} \coloneqq \mr{A}\otimes_{\mr{R}} \mr{Q}$ and similarly for $\mr{M}_{\mr{Q}}$. Then we have
 \begin{enumerate}
  \item $\mr{C}^{p,q}(\mr{A_Q,M_Q}) = \mr{C}^{p,q}(\mr{A,M})\otimes_{\mr{R}}\mr{Q}$.
  \item Assuming $\mr{Q}$ flat over $\mr{R}$, $\mr{HH}^{p,q}(\mr{A_Q,M_Q}) = \mr{HH}^{p,q}(\mr{A,M})\otimes_{\mr{R}}\mr{Q}$.
 \end{enumerate}
\end{proposition}
\2
\begin{proposition}\label{extraformalityres}
 Suppose that $\mr{R}$ is Noetherian. Consider a graded algebra $\mr{A}$ over $\mr{R}$ and a graded module $\mr{M}$ over $\mr{A}$. Assume that $\mr{A}$, $\mr{M}$ are finite projective and that for all $p,q \in \mathbf{Z}$ the $\mr{R}$-module $\mr{HH}^{p,q}(\mr{A},\mr{M})$ is projective. Then, for any morphism of commutative rings $\mr{R} \to \mr{Q}$, we have $$\mr{HH}^{p,q}(\mr{A_Q,M_Q}) = \mr{HH}^{p,q}(\mr{A,M}) \otimes_{\mr{R}}\mr{Q}.$$ 
\end{proposition}

\paragraph{Obstruction theory.}In this paragraph we formulate obstruction theory for $\mr{A}_n$-modules. Our goal is to apply it to a problem where we are extending an $\mr{A}_n$-morphism $f$ and we only care for keeping $f_1$ fixed. In particular, a theory for obstructions where the last component is allowed to vary is good enough for us and that's what we develop here. There are more general versions where one allows the last $r$ components to vary for some $r < n$. Compared to the case where we keep all components fixed, our approach has the advantage that the corresponding obstructions are cohomology classes rather than equations in the space Hochschild cochains.\\ 
$\hspace*{3mm}$We show how the Hochschild cohomology defined in the previous section controls the obstructions to extending $\mr{A}_n$-modules to $\mr{A}_{n+1}$-modules by allowing the last multiplication to vary as well as $\mr{A}_n$-morphisms to $\mr{A}_{n+1}$-morphisms, varying the last component. We focus on modules, but analogous versions for algebras can also be formulated.
\2
\begin{proposition}
 Let $\mr{A}$ be a minimal $\mr{A}_n$-algebra. Let $\mr{M}$ be a graded $\mr{R}$-module and suppose $$m^{\mr{M}}_i : \mr{M}\otimes \mr{A}^{i-1} \to \mr{M}, \quad 2\le i \le n+1,$$ are graded morphisms of degree $2-i$. Assume that for $1\le i \le n$ the $m_i^{\mr{M}}$ define an $\mr{A}_{n}$-structure on $\mr M$. Then the subexpression of $(*'_{n+1})$ given by $$\sum_{j+1+l,k \not=1,2} (-1)^{jk+l}m_{j+1+l}(\mr{id}^{\otimes j}\otimes m_{k} \otimes \mr{id}^{\otimes l})$$ defines a cocycle in $(\mr{C}(\mr{A}(2),\mr{M}(2));\mr{d})$ which we denote by $\mr{c}(m_{3}^{\mr{M}},\cdots,m_{n}^{\mr M})$ and equation $(*'_{n+1})$ becomes $$\mr{c}(m_{3}^{\mr{M}},\cdots,m_{n-1}^{\mr M}) + \mr{d}(m_{n}^{\mr M})=0.$$
\end{proposition}
\2
\begin{proposition}\label{obs}
 Let $\mr{A}$ be a minimal $\Aa$-algebra. Let $\mr M$ and $\mr{N}$ be two minimal $\Aa$-modules over $\mr{A}$. Suppose given $$f_{i} : \mr{M}\otimes \mr{A}^{\otimes i-1} \to \mr{N}, \quad 1\le i\le n+1,$$ of degree $1-i$ such that the morphisms $f_{i}$, for $1\le i \le n$, define an $\mr{A}_{n}$-morphism. The subexpression of $(**'_{n+1})$$$\sum_{k\not=1,2} (-1)^{jk+l}f_{j+1+l}(\mr{id}^{\otimes j}\otimes m_{k} \otimes \mr{id}^{\otimes l}) - \sum_{s\not=0,1}m_{s+1}(f_{r}\otimes \mr{id}^{\otimes s})$$ defines a cocycle in $(\mr{C}(\mr{A}(2),\mr{M}(2),\mr{N}(2));\mr{d})$, denoted by $\mr{c}(f_{1},\cdots,f_{n-1})$. Then equation $(**'_{n+1})$ becomes $$\mr{c}(f_1,\cdots,f_{n-1}) + \mr{d}(f_{n})=0.$$
\end{proposition}

\paragraph{Formality of $\Aa$-modules.}
\begin{definition}
 Let $\mr{A}$ be an $\Aa$-algebra. Let $n \in \mathbf{N}\cup \{\infty\}$ and assume that $\mr{A}$ is $\mr{A}_n$-formal. We say that an $\Aa$-module $\mr M$ over $\mr A$ is $\mr{A}_n$-formal if there exists an $\mr{A}_n$-quasi-isomorphism of pairs $(\mr{HA},\mr{HM}) \to (\mr{A},\mr{M})$, where $\mr{HA}$ (resp. $\mr{HM}$) is the ordinary graded associative algebra (resp. module).
\end{definition}
\2
\begin{remark}\leavevmode
\begin{enumerate}
\item  In other words $\mr{M}$ is $\mr{A}_n$-formal if it admits a minimal $\mr{A}_n$-model with vanishing higher multiplications. 
\item Notice that our definition of $\mr{A}_n$-formality for a module $\mr{M}$ over $\mr{A}$ assumes that the $\Aa$-algebra $\mr{A}$ is $\mr{A}_n$-formal.
\item When $n =\infty$, we drop the $\Aa$ and simply say that $\mr{M}$ is formal.
\item By \cref{dgainf}, under the usual projectivity assumptions, a differential graded module $\mr{M}$ over a formal differential graded algebra $\mr{A}$ is differential graded formal iff its minimal model $\mr{HM}$ is formal as an $\Aa$-module over the (formal) $\Aa$-algebra $\mr{HA}$.
\end{enumerate}
\end{remark}
\2
$\hspace*{3mm}$Recall that if $\mr{A}$ is a minimal $\Aa$-algebra, then $\mr{A}(2)$ stands for the underlying graded associative algebra, so if $\mr{A}$ is just a graded associative algebra, then $\mr{A}=\mr{A}(2)$.\\
$\hspace*{3mm}$Similarly, let $\mr{M}$ be a minimal $\Aa$-module over a graded algebra $\mr{A}$. Then $\mr M(2)$ denotes the $\Aa$-module on the graded space $\mr M$ with structure morphisms given by $$m_2^{\mr M(2)}= m_{2}^{\mr M}  \text{ and } m_i^{\mr M(2)}=0 \text{ for } i \not=2,$$ i.e. $\mr{M}(2)$ is the underlying graded associative $\mr A$-module, considered as an $\Aa$-module - this is possible precisely because $\mr{M}$ is minimal and $\mr{A}$ is just a graded associative algebra, so has no higher multiplications, and the $\Aa$-relations for $\mr{M}(2)$ reduce to associativity. In this notation, $\mr{M}$ is a formal $\Aa$-module if it is quasi-isomorphic to $\mr{M}(2)$.
\2
\begin{remark}
 Let $\mr M$ be a minimal $\Aa$-module over a graded algebra $\mr A$. If $\mr M$ is formal, there exists a quasi-isomorphism $f : \mr{M} \to \mr{M(2)}$ lifting the identity on the underlying complexes.
\end{remark}
\2
\begin{proposition}\label{modformality}
 Let $\mr{R}$ be an integral domain with field of fractions $k(\eta)$. Let $\mr A$ be a graded $\mr{R}$-algebra and let $\mr{M}$ be a minimal $\Aa$-module over $\mr A$ such that both $\mr A$ and $\mr M$ are finite projective as $\mr{R}$-modules. Assume that $\mr{HH}^{n,1-n}(\mr{A},\mr{M}(2))$ is torsion-free for all $n$. If $\mr {M}_{\eta}= k(\eta)\otimes_{\mr{R}} \mr{M}$ is formal, then $\mr{M}$ is formal and there exists a quasi-isomorphism $f: \mr{M} \to \mr M(2) $ lifting the identity on the underlying complexes. In particular, the fibres $\mr{M}_{\mathfrak{p}}$ are formal for all $\mathfrak{p} \in \mr{Spec}(\mr{R})$.
\end{proposition}
\begin{proof}
 We are going to construct the quasi-isomorphism $f$ inductively. Since $\mr M$ and $\mr M(2)$ have the same underlying $\mr A_2$-modules, we set $f_1=\mr{id}$, and for any $f_2$, we have a quasi-isomorphism $(\mr{id},f_2) : \mr{M} \to \mr{M}(2)$ of $\mr{A}_2$-modules over $\mr{A}$. Then $[\mr{c}(\mr{id})]= [m_3]$ is the Massey product which vanishes generically since $\mr{M}_\eta$ is formal. Since $\mr{HH}^{2,-1}(\mr{A},\mr{M}(2))$ is torsion-free, it follows that $[\mr{c}(\mr{id})]$ vanishes everywhere. Hence, we can choose $f_2$ such that $$\mr{c}(\mr{id}) + \mr{d}(f_2)=0.$$ 
 $\hspace*{3mm}$There exists an $\Aa$-module structure $\mr{M}^{(2)} = (\mr{M},m^{(2)})$ on the graded $\mr{R}$-module $\mr{M}$ such that 
 \begin{equation*}
  m^{(2)}_2=m_{2}^{\mr{M}}, m^{(2)}_3=0 \text{ and }
  \tilde{f}_{2} \coloneqq (\mr{id},f_2,0,\cdots) : \mr{M} \to \mr{M}^{(2)}
 \end{equation*}
 is a quasi-isomorphism of $\Aa$-modules. Indeed, $\mr{M}^{(2)}$ can be constructed inductively as follows: we treat the multiplications $m^{(2)}$ of $\mr{M}^{(2)}$ as variables and require that $(\mr{id},f_2,0,\cdots)$ defines a quasi-isomorphism of $\Aa$-modules. The corresponding equations for $m^{(2)}$ can be uniquely solved as they are all of the form $$\sum_{\substack{j+k+l=n, k\ge2}} (-1)^{jk+l}f_{j+1+l}(\mr{id}^{\otimes j}\otimes m_{k}\otimes \mr{id}^{\otimes l}) = m^{(2)}_{n}+ m^{(2)}_{n-1}(f_{2}\otimes \mr{id}^{\otimes n-2}),$$ where $j+1+l= 1 \text{ or } 2$. For example, $n=2$ implies that $m^{(2)}_{2}=m_{2}$ and $n=3$ is just the equation $\mr{c}(\mr{id}) + \mr{d}(f_2)=m^{(2)}_3$, i.e. $m^{(2)}_3=0$.\\
$\hspace*{3mm}$Assume by induction that we have constructed $f_2,\cdots, f_n$ and $\mr{M}^{(2)},\cdots,\mr{M}^{(n)}$ with \begin{equation*}
 m_{2}^{(i)}=m_{2}^{\mr{M}}  \text{ and } m^{(i)}_{j}=0 \text{ for all } 3\le j \le i+1 
 \end{equation*} 
 such that, for all $2\le i \le n$, the maps
 \begin{equation*}
  \tilde{f}_{i} \coloneqq (\mr{id},0,\cdots,0,f_i,0,\cdots) : \mr{M}^{(i-1)} \to \mr{M}^{(i)}
 \end{equation*}
are quasi-isomorphisms of $\Aa$-modules.\\
$\hspace*{3mm}$In order to construct the pair $(f_{n+1},\mr{M}^{(n+1)})$, it suffices to show that there exists $f_{n+1}$ such that $$(\mr{id},0,\cdots,0,f_{n+1},0) :\mr{M}^{(n)} \to \mr{M}^{(n)}(2)$$ is an $\mr{A}_{n+2}$-morphism, i.e. $\mr{A}_{n+2}$-formality of $\mr{M}^{(n)}$. We know that $$(\mr{id},0,\cdots,0,f_{n+1}):\mr{M}^{(n)} \to \mr{M}^{(n)}(2)$$ is an $\mr{A}_{n+1}$-morphism for any $f_{n+1}$. The goal is to show $f_{n+1}$ can be chosen so that in fact any $f_{n+2}$ would give an $\mr{A}_{n+2}$-morphism $\mr{M}^{(n)} \to \mr{M}^{(n)}(2)$.\\
$\hspace*{3mm}$By \cref{obs} it suffices that the cohomology class $[\mr{c}(\mr{id},0,\cdots,0)]$ vanish. Notice that the function $$(g_2,\cdots,g_n) \mapsto \mr{c}(\mr{id},g_2,\cdots,g_n)$$ is constant precisely because the corresponding higher multiplications in $\mr{M}^{(n)}$ vanish, and the same applies for the localised version $\mr{M}^{(n)}_{\eta}$. Since $\mr M_{\eta}$, and hence also $\mr{M}^{(n)}_{\eta}$, is formal, we see that $$[\mr{c}(\mr{id},0,\cdots,0)]_{\eta}=0.$$ Furthermore, $\mr{HH}^{n+1,-n}(\mr{A},\mr{M}(2))$ is torsion-free, so we get that $[\mr{c}(\mr{id},0,\cdots,0)]=0,$ hence there exists $f_{n+1}$ such that $$(\mr{id},0,\cdots,0, f_{n+1},f_{n+2}):\mr{M}^{(n)} \to \mr{M}^{(n)}(2)$$ is an $\mr{A}_{n+2}$-morphism for any choice of $f_{n+2}.$\\
$\hspace*{3mm}$Then, as above, the standard construction gives an $\Aa$-module $\mr{M}^{(n+1)}$ such that
\begin{equation*}
 m_{2}^{(n+1)}=m_{2}^{\mr{M}}  \text{ and } m^{(n+1)}_{j}=0 \text{ for all } 3\le j \le n+2 
 \end{equation*} 
and the map
 $ \tilde{f}_{n+1}\coloneqq (\mr{id},0,\cdots,0,f_{n+1},0,\cdots) : \mr{M}^{(n)} \to \mr{M}^{(n+1)}$
is a quasi-isomorphism of $\Aa$-modules.\\
$\hspace*{3mm}$The infinite composition $$f \coloneqq \cdots \circ \tilde{f}_{n}\circ \cdots \circ \tilde{f}_{2} : \mr{M} \to \mr{M}(2)$$ defines the required quasi-isomorphism and we note that it is well-defined because composition with $\tilde{f}_{n}$ leaves the components in weights $i < n$ fixed.
\end{proof}
$\hspace*{3mm}$In fact the proof shows also that the following proposition holds.
\2
\begin{proposition}
 Let $\mr{A}$ be a graded algebra. An $\Aa$-module $\mr{M}$ over $\mr{A}$ is formal if and only if it is $\mr{A}_n$-formal for all $n \in \mathbf{N}$.
\end{proposition}
\2
\begin{corollary}
 Suppose that $\mr{R}$ is Noetherian. Let $\mr{A}$ be a graded algebra over $\mr{R}$ and let $\mr{M}$ be a minimal $\Aa$-module over $\mr{A}$. Suppose that $\mr{A}$ and $\mr{M}$ are finite projective $\mr{R}$-modules and $\mr{HH}^{p,q}(\mr{A},\mr{M}(2))$ is projective for all $p,q \in \mathbf{Z}$. Then the set $$\mr{F}(\mr{M})=\{\mathfrak{p}\in \mr{Spec}(\mr{R}) \enskip | \enskip \text{the } \Aa \text{-module } \mr{M}_x \text{ is formal }\}$$ is closed under specialisation.
\end{corollary}
\begin{proof}
 Assume $\mr{F}(\mr{M})$ is non-empty. Let $\mathfrak{p} \in \mr{F}(\mr{M})$ and consider its closure $\bar{\mathfrak{p}} = \mr{Spec}(\mr{Q}) \subset \mr{Spec}(\mr{R})$. The ring $\mr{Q}$ is an integral domain and the base-changes $\mr{A}_{\mr{Q}} = \mr{A}\otimes_{\mr{R}}\mr{Q}$ and $\mr{M}_{\mr{Q}} = \mr{M}\otimes_\mr{R} \mr{Q}$ are finite projective over $\mr{Q}$. Furthermore, by \cref{extraformalityres} we have $$\mr{HH}^{n,1-n}(\mr{A}_{\mr{Q}},\mr{M}(2)_{\mr{Q}}) = \mr{HH}^{n,1-n}(\mr{A},\mr{M}(2))\otimes_{\mr{R}}\mr{Q},$$ so $\mr{HH}^{n,1-n}(\mr{A}_{\mr{Q}},\mr{M}(2)_{\mr{Q}})$ is projective over $\mr{Q}$, in particular torsion-free. Hence $\mr{M}_\mr{Q}$ is formal by \cref{modformality}.
\end{proof}
\begin{proposition}
  Suppose that $\mr{R}$ is Noetherian and $\mr{I} \subset \mr{R}$ is an ideal such that $\cap_{l}\mr{I}^{l}=0$. Let $\mr{A}$ be a graded algebra over $\mr{R}$ and let $\mr{M}$ be a minimal $\Aa$-module over $\mr{A}$. Suppose that $\mr{A}$ and $\mr{M}$ are finite projective $\mr{R}$-modules and that $\mr{HH}^{p,q}(\mr{A},\mr{M}(2))$ is projective for all $p,q \in \mathbf{Z}$. If $\mr{M}/\mr{I}^{l}$ is a formal $\Aa$-module over the graded algebra $\mr{A}/\mr{I}^{l}$ for all $l \in \mathbf{N}$, then $\mr{M}$ is formal.
\end{proposition}
\begin{proof}
 Since $\mr{M}$ is projective over $\mr{R}$ and minimal, it is $\mr{A}_2$-formal. Then, as in the proof of \cref{modformality}, to show it is $\mr{A}_3$-formal, it is enough to prove that the Massey class $[m_3]$ vanishes in $\mr{HH}^{2,-1}(\mr{A,M(2)})$. Using \cref{extraformalityres} we have $$\mr{HH}^{p,q}(\mr{A/I}^l,\mr{M(2)/I}^l) =\mr{HH}^{p,q}(\mr{A},\mr{M(2)})\otimes_\mr{R}\mr{R}/\mr{I}^{l} \text{ for all } p,q \in \mathbf{Z}.$$ By assumption $[m_3]\otimes 1=0$ in $\mr{HH}^{2,-1}(\mr{A/I}^l,\mr{M(2)/I}^l)$ for all $l\ge 1$. There is an exact sequence $ 0 \to \cap_{l}\mr{I}^l \to \mr{R} \to \prod_{l} \mr{R}/\mr{I}^l$, since $\cap_l \mr{I}^l =0$ and $\mr{HH}^{p,q}(\mr{A},\mr{M(2)})$ is projective over $\mr{R}$, we conclude that $[m_3]=0$. By induction $\mr{M}$ is $\mr{A}_n$-formal for all $n\in \mathbf{N}$, done.
\end{proof}
\begin{proposition}
 Suppose that $\mr{R}$ is Noetherian with trivial Jacobson radical $\mathfrak{J}(\mr{R})$. Let $\mr{A}$ be a graded associative algebra over $\mr{R}$, and let $\mr{M}$ be a minimal $\Aa$-module over $\mr{A}$. Suppose that $\mr{A}$ and $\mr{M}$ are finite projective $\mr{R}$-modules and that $\mr{HH}^{p,q}(\mr{A},\mr{M}(2))$ is projective for all $p,q \in \mathbf{Z}$. If $\mr{M}_\mathfrak{m}$ is formal for all closed points $\mathfrak{m} \in \mr{Spec}(\mr{R})$, then $\mr{M}$ is formal. In particular, $\mr{M}_\mathfrak{p}$ is formal for all $\mathfrak{p} \in \mr{Spec}(\mr{R})$.
\end{proposition}
\begin{proof}
 As in the previous proposition, we prove that $\mr{M}$ is $\mr{A}_n$-formal for all $n\in \mathbf{N}$. The proof is exactly the same using the exact sequence $0\to\mathfrak{J}(\mr{R}) \to \mr{R}\to \prod_{\mathfrak{m}}\mr{R}/\mathfrak{m}$, \cref{extraformalityres} and projectivity of $\mr{HH}^{p,q}(\mr{A},\mr{M(2)})$.
\end{proof}

\section{Perverse sheaves and DQ modules}\label{sec2}

$\hspace*{3mm}$In this section we shall briefly review some results on DQ-modules and perverse sheaves on Lagrangian intersections we need for applications. This is mainly to fix notation, in particular we refer the reader to the original papers for proofs. The material on perverse sheaves on Lagrangian intersections and, more generally, on $\mr{d}$-critical loci is due to Joyce et al. and we refer to \cite{bbdjs}, \cite{MR3399099}. A good general reference for DQ-modules, which we also closely follow, is Kashiwara and Schapira \cite{MR3012169}. 

\paragraph{Perverse sheaves on $\mr{d}$-critical loci.} The local setting, we are interested in, is as follows: for a complex manifold $\mr{X}$ and a function $f$ on $\mr{X}$, we consider the intersection $\mr{crit}f \coloneqq \mr{X} \cap \Gamma_{\mr{d}f}$, where $\Gamma_{\mr{d}f} \subset \Omega_{\mr{X}}$ is the Lagrangian, given by the graph of $\mr{d}f \in \Gamma(\mr{X},\Omega_{\mr{X}})$. In particular, if $\mr{X}$ is the cotangent bundle of a complex manifold $\mr{M}$, then for any function $f$ on $\mr{M}$, we can consider the Lagrangian $\mr{L}= \Gamma_{\mr{d}f}$, so $\mr{crit}f=\mr{L}\cap \mr{M}$, i.e. in this case the intersection of the Lagrangians can be written as the critical locus of $f$. It turns out that this remains true locally in the general case where $\mr{X}$ is a holomorphic symplectic and $\mr{L}$, $\mr{M}$ are two Lagrangians. \\
$\hspace*{3mm}$On $\mr{crit}f \subset X$, we have a naturally defined perverse sheaf of vanishing cycles $\mathscr{P}_{\mr{X},f}$ which is the image of $\C_{\mr{X}}[\mr{dim}\,\mr{X}]$ under the vanishing cycles functor over the critical values of $f$.\\
$\hspace*{3mm}$We are going to consider the global versions of the constructions discussed above. Let $\mr{X}/\C$ be a scheme. Suppose we are given an embedding of $\mr{X}$ into a smooth scheme $\mr{S}$ with ideal sheaf $\mathscr{I}$, then we define the complex of derived $1$-jets $$\mathbf{J}^{1}_\mr{X} \coloneqq \mathscr{O}_{\mr{S}}/\mathscr{I}^{2} \to \Omega_{\mr{S}}|_{\mr{X}},$$ in degrees $-1$ and $0$. It can be shown that it is independent of the embedding and is naturally quasi-isomorphic to the cone of the de Rham differential $\mathscr{O}_\mr{X} \to \mathbf{L}_\mr{X}$, where $\mathbf{L}_\mr{X} \coloneqq \mathscr{I}/\mathscr{I}^{2} \to \Omega_{\mr{S}}|_\mr{X}$ is the truncated cotangent complex. We are interested in the sheaf $$\mathscr{S}_{\mr{X}} = \mathscr{H}^{-1}(\mathbf{J}^{1}_\mr{X}).$$ 
\begin{definition}(Joyce \cite{MR3399099})
 A structure of a $\mr{d}$-critical locus on $\mr{X}$ is a choice of $s \in \Gamma(\mr{X},\mathscr{S}_\mr{X})$ such that for any $x \in \mr{X}$ there exists an open $\mr{U}$, containing $x$, and a closed embedding of $\mr{U}$ into a smooth $\mr{S}$ together with a function $f$ on $\mr{S}$ such that $s|_{\mr{U}} = f$ in $\Gamma(\mr{U},\mathscr{O}_{\mr{S}}/\mathscr{I}^{2})$ and $\mr{U} = \mr{crit}f \subset \mr{S}$. The triple $(\mr{U},\mr{S},f)$ is a chart for the $\mr{d}$-critical locus $\mr{X}$.
\end{definition}
\2
$\hspace*{3mm}$Let $(\mr{X},s)$ be a $\mr{d}$-critical locus. Then, given a chart $(\mr{U},\mr{S},f)$, we can consider the canonical bundle $\mr{K}_\mr{S}|_{\mr{U}_\mr{red}}$ and ask whether we can glue these line bundles for a covering by critical charts. The answer is no, but we can glue their squares $\mr{K}_\mr{S}^{\otimes 2}|_{\mr{U}_\mr{red}}$ to get a line bundle on $\mr{X}_\mr{red}$. Suppose that $\mr{X}$ is of the form $\mr{crit}f$, then the obstruction is a $\pm1$-cocycle, hence we can glue the squares to get a line bundle on $\mr{X}_\mr{red}$.
\2
\begin{proposition}(Joyce \cite{MR3399099})
 Let $(\mr{X},s)$ be a $\mr{d}$-critical locus. There exists a unique line bundle $\mr{K}_{(\mr{X},s)}$ on $\mr{X}_\mr{red}$ such that for any chart $(\mr{U},\mr{S},f)$ we have an isomorphism $$\lambda_{(\mr{U},\mr{S},f)}:\mr{K}_{(\mr{X},s)}|_{\mr{U}_\mr{red}} \cong \mr{K}_{\mr{S}}^{\otimes 2}|_{\mr{U}_\mr{red}}$$ such that for any étale morphism $\varphi : (\mr{U},\mr{S},f) \to (\mr{V},\mr{T},g)$ of charts, i.e. $\varphi : \mr{S} \to \mr{T}$ is étale, $\varphi|_{\mr{U}} : \mr{U} \xhookrightarrow{} \mr{V}$ is the canonical inclusion and $f = g\circ \varphi$, we have $$\lambda_{(\mr{U},\mr{S},f)} = \mr{det}(\mr{d}\varphi)^{\otimes 2}|_{\mr{U}_\mr{red}}\circ \lambda_{(\mr{V},\mr{T},g)}|_{\mr{U}_{\mr{red}}}.$$
\end{proposition}
\begin{example}
 If $\mr{X}$ is smooth, then $\mr{K}_{(\mr{X},0)} = \mr{K}_{\mr{X}}^{\otimes 2}$. We get an extra $\mr{K}_\mr{X}$ factor because, as a derived scheme, the critical locus $\mr{crit}(0:\mr{X} \to \C)$ is the shifted cotangent bundle $\Omega_\mr{X}[1]$.
\end{example}

\2
\begin{theorem}\label{pervdcrit}(Brav et al. \cite{bbdjs})
 Let $(\mr{X},s)$ be a $\mr{d}$-critical locus. Assume that its canonical bundle $\mr{K}_{(\mr{X},s)}$ admits a square root, called orientation of $(\mr{X},s)$. Then there exists a perverse sheaf $\mathscr{P}_{(\mr{X},s)}$ on $\mr{X}$ such that if $(\mr{U},\mr{S},f)$ is a chart, then we have a natural isomorphism $$\mathscr{P}_{(\mr{X},s)}|_{\mr{U}} \simeq \mathscr{P}_{\mr{S},f}\otimes_{\C}\mathfrak{K}_{\mr{or}},$$ where $\mathfrak{K}_{\mr{or}}$ is the local system associated to $\mr{K}_{(\mr{X},s)}^{-1/2}|_{\mr{U_{red}}}\otimes \mr{K}_{\mr{S}}|_{\mr{U_{red}}}$.
\end{theorem}
\2
\begin{proposition}(Bussi \cite{Bussi:2014psa})
 Let $\mr{X}$ be a holomorphic symplectic variety. Suppose given two Lagrangians $\mr{L}$ and $\mr{M}$ in $\mr{X}$. Then then intersection $\mr{L}\cap \mr{M}$ admits a structure of a $\mr{d}$-critical locus $(\mr{L}\cap \mr{M},s)$ with canonical bundle $\mr{K}_{(\mr{L}\cap\mr{M},s)}=\mr{K}_\mr{L}|_\mr{\mr{L}\cap\mr{M}_{red}}\otimes\mr{K}_\mr{M}|_{\mr{\mr{L}\cap\mr{M}_{red}}}$.
\end{proposition}
\2
\begin{corollary}
 Consider the $\mr{d}$-critical locus $(\mr{L}\cap \mr{M},s)$ and assume that $\mr{K}_\mr{L}|_\mr{\mr{L}\cap\mr{M}_{red}}\otimes\mr{K}_\mr{M}|_{\mr{\mr{L}\cap\mr{M}_{red}}}$ admits a square root. Then there exists a perverse sheaf $$\mathscr{P}_{\mr{L,M}} \simeq \mathscr{P}_{\mr{M,L}}$$ on $\mr{L}\cap \mr{M}$ with the properties described in \cref{pervdcrit}.
\end{corollary}
\2
\begin{lemma}\label{lemmasmoothcase}
  Let $\mr{L,M}$ be two Lagrangians intersecting cleanly, then there is an isomorphism $$\mr{K_{\mr{L}\cap\mr{M}}\otimes K_{\mr{L}\cap\mr{M}} \cong \left.K_{L}\right|_{\mr{L}\cap\mr{M}}\otimes \left.K_{M}\right|_{\mr{L}\cap\mr{M}}}.$$
 \end{lemma}

\begin{corollary}
 Let $\mr{L}\cap \mr{M}$ be smooth. Then $(\mr{L}\cap \mr{M},s)$ is oriented and for any choice of $\mr{K}_{(\mr{L}\cap \mr{M},s)}^{1/2}$ we have $\mathscr{P}_{\mr{L,M}} = \mathfrak{K}_{\mr{or}}[\mr{dim}\,\mr{X}]$, where $\mathfrak{K}_{\mr{or}}$ is the local system associated to $\mr{K}_{(\mr{L}\cap \mr{M},s)}^{-1/2} \otimes \mr{K}_{\mr{L}\cap \mr{M}}$.
\end{corollary}

\paragraph{DQ-algebras.}
We let $\mr{X}$ be a complex manifold. Let $\C\brak$ be the ring of formal power series in $\hbar$, and $\C\cbrak$ its field of fractions, i.e. the field of formal Laurent series. Define a sheaf of $\C\brak$-algebras: $$\mathscr{O}_\mr{X}[\![\hbar]\!] = \varprojlim \mathscr{O}_\mr{X} \otimes_{\C} \C[\![\hbar]\!]/\hbar^{n}.$$
\begin{definition}
 A star product on $\mathscr{O}_\mr{X}[\![\hbar]\!]$ is a $\C\brak$-bilinear associative multiplication $\star$ such that $$f\star g = \sum_{i\ge 0} \mr{P}_i(f,g)\hbar^{i}, \text{ where } f,g \in \mathscr{O}_\mr{X},$$ such that $\mr{P}_i$ are holomorphic bidifferential operators with $\mr{P}_{0}(f,g)=fg$ and $\mr{P}_{i}(f,1) = \mr{P}_i(1,f)=0$ for all $i \ge 1$. The pair $(\mathscr{O}_\mr{X}[\![\hbar]\!],\star)$ is called a star algebra.
\end{definition}
\2
\begin{definition}
 A deformation quantisation algebra (DQ-algebra) on a complex manifold $\mr{X}$ is a sheaf of $\C\brak$-algebras $\mathscr{A}_\mr{X}$ locally isomorphic to a star algebra as a $\C\brak$-algebra.
\end{definition}
\2
\begin{example}\label{Poissonexample}
 Let $\mathscr{A}_\mr{X}$ be a DQ-algebra on $\mr{X}$. Let $\pi : \mathscr{A}_\mr{X} \to \mathscr{A}_\mr{X}/\hbar\mathscr{A}_\mr{X} \cong \mathscr{O}_\mr{X}$. For any $f,g \in \mathscr{O}_\mr{X}$, choose lifts $\tilde{f}, \tilde{g}$ such that $\pi(\tilde{f})=f$ and $\pi(\tilde{g})=g$. Then define a bracket $$\{f,g\} = \pi(\hbar^{-1}(\tilde{f}\tilde{g}-\tilde{g}\tilde{f})).$$ This is independent of the choices made and defines a Poisson structure on $\mr{X}$.
\end{example}
\2
\begin{example}
 Let $\mr{X}$ be a complex manifold. The cotangent bundle $\Omega_\mr{X}$ supports a filtered sheaf of $\C$-algebras $\hat{\mathscr{E}}_{\Omega_\mr{X}}$ of formal microdifferential operators. We start by recalling its definition. Fix $(x_1,\cdots,x_n)$ coordinates on $\mr{X}$, and write $(x_1,\cdots,x_n,\xi_1,\cdots,\xi_n)$ for the induced coordinates on $\Omega_\mr{X}$. Let $\mathscr{O}_{\Omega_\mr{X}}(m)$ be the sheaf of homogeneous functions in the fibre coordinates on $\Omega_\mr{X}$ of degree $m$, i.e. $$\left(\sum \xi_j\partial/\partial\xi_j - m\right)f(x,\xi)=0.$$ 
 $\hspace*{3mm}$We define the sheaf of formal microdifferential operators of order $\le m$ by $$\hat{\mathscr{E}}_{\Omega_\mr{X}}(m) = \prod_{j \in \mathbf{N}} \mathscr{O}_{\Omega_\mr{X}}(m-j).$$ In order to get a sheaf globally on $\Omega_\mr{X}$, we glue these sheaves on overlaps using the transformation rule for total symbols of differential operators. Taking the limit over $m \in \mathbf{Z}$, we get the sheaf of formal microdifferential operators on $\Omega_\mr{X}$: $$\hat{\mathscr{E}}_{\Omega_\mr{X}} = \varinjlim_{m \in \mathbf{Z}}\hat{\mathscr{E}}_{\Omega_\mr{X}}(m).$$ 
 $\hspace*{3mm}$Note that there are products $\hat{\mathscr{E}}_{\Omega_\mr{X}}(l)\otimes_{\C}\hat{\mathscr{E}}_{\Omega_\mr{X}}(m) \to \hat{\mathscr{E}}_{\Omega_\mr{X}}(l+m)$, given by $$(f\star g)_{l}(x,\xi) = \sum_{\substack{l=i+j-|\alpha|\\ \alpha \in \mathbf{N}}}\frac{1}{\alpha!}(\partial/\partial\xi_1)^{\alpha_1}\cdots (\partial/\partial\xi_n)^{\alpha_n}f_i(x,\xi)\cdot (\partial/\partial x_1)^{\alpha_1}\cdots (\partial/\partial x_n)^{\alpha_n}g_j(x,\xi).$$ In particular $\hat{\mathscr{E}}_{\Omega_\mr{X}}$ and $\hat{\mathscr{E}}_{\Omega_\mr{X}}(0)$ are sheaves of (non-commutative) $\C$-algebras. Notice that the total symbol of a differential operator is a polynomial in $\xi_1,\cdots, \xi_n$, so essentially we are just allowing symbols which are general holomorphic functions rather than just polynomials.\\
 $\hspace*{3mm}$Let $t$ be the coordinate on $\C$ and $(t;\tau)$ - the symplectic coordinates on $\Omega_{\C}$. Let $\Omega_{\mr{X}\times \C,\tau\not=0}$ be the open subset of $\Omega_{\mr{X}\times \C}$ where $\tau\not=0$. We have a map $$\rho : \Omega_{\mr{X}\times \C,\tau\not=0} \to \Omega_{\mr{X}}, \quad (x,t;\xi,\tau) \mapsto (x,\tau^{-1}\xi).$$ Define the subsheaf of operators independent of $t$: $$\hat{\mathscr{E}}_{\Omega_{\mr{X}\times \C},\hat{t}}(0)= \{\mr{P} \in \hat{\mathscr{E}}_{\Omega_{\mr{X}\times \C}}(0) \text{ such that } [\mr{P},\partial_{t}]=0\}.$$ Then, letting $\hbar$ act as $\tau^{-1}$, we define the canonical DQ-algebra on $\Omega_{\mr{X}}$ by $$\hat{\mathscr{W}}_{\mr{X}}(0) = \rho_{*}\hat{\mathscr{E}}_{\Omega_{\mr{X}\times \C},\hat{t}}(0).$$ The $\hbar$-localisation of $\hat{\mathscr{W}}_{\mr{X}}(0)$ is denoted by $\hat{\mathscr{W}}_{\mr{X}}$.
\end{example}
\2
\begin{definition}
 Let $\mr{X}$ be a topological space. An $\mr{R}$-algebroid on $\mr{X}$ is an $\mr{R}$-linear stack $\mathscr{A}$ which is locally non-empty and any two objects in $\mathscr{A}(\mr{U})$ are locally isomorphic for any open $\mr{U} \subset \mr{X}$.
\end{definition}
\2
\begin{example}Fix a topological space $\mr{X}$.
 Let $\mathscr{A}$ be a sheaf of $\mr{R}$-algebras on $\mr{X}$. We consider the prestack $\mr{U} \to \mathscr{A}(\mr{U})^{+}$, where $\mathscr{A}(\mr{U})^{+}$ is the $\mr{R}$-linear category with one object whose endomorphisms are given by $\mathscr{A}(\mr{U})$. The associated stack is denoted by $\mathscr{A}^{+}$. It's an $\mr{R}$-algebroid.\\
 $\hspace*{3mm}$Conversely, suppose that $\mathscr{A}$ is an algebroid. If $\mathscr{A}(\mr{X})$ is non-empty, choose any $\tau \in \mathscr{A}(\mr{X})$. We have an equivalence $\mathscr{A} \simeq \shom(\tau,\tau)^{+}$ of $\mr{R}$-algebroids.
\end{example}
\2
$\hspace*{3mm}$Given an $\mr{R}$-algebroid $\mathscr{A}$ over $\mr{X}$, let $\mathscr{M}_\mr{R}$ be the stack of sheaves of $\mr{R}$-modules on $\mr X$, we define the $\mr{R}$-linear abelian category of modules over $\mathscr{A}$ by $$\mr{Mod}(\mathscr{A}) = \mr{Fct}(\mathscr{A},\mathscr{M}_\mr{R}).$$ 
\begin{definition}
 A deformation quantisation algebroid (DQ-algebroid) on $\mr{X}$ is a $\C\brak$-algebroid $\mathscr{A}$ such that, for any open $\mr{U}\subset \mr{X}$ and $\tau \in \mathscr{A}(\mr{U})$, the $\C\brak$-algebra $\mathscr{H}om(\tau,\tau)$ is a DQ-algebra on $\mr{U}$.
\end{definition}
\2
\begin{remark}
 If $\mr{X}$ is a holomorphic symplectic variety, then the holomorphic Darboux theorem implies that locally we have canonical DQ-algebras associated with $\mr{X}$, but they won't generally glue to a global DQ-algebra. In general, one has to twist them by "half forms", i.e. by the twisted sheaf of half top forms and its dual. It is a theorem of Polesello and Schapira \cite{10.1155/S1073792804132819} that the corresponding twisted DQ-algebras glue and we obtain a DQ-algebroid still denoted $\hat{\mathscr{W}}_{\mr{X}}(0)$.\\
 $\hspace*{3mm}$Any other DQ-algebroid $\mathscr{A}_{\mr{X}}$ on $\mr{X}$ will be equivalent to $\hat{\mathscr{W}}_{\mr{X}}(0)\otimes_{\C\brak}\mathscr{L}$ for some invertible $\C\brak$-algebroid $\mathscr{L}$. Hence DQ-algebroids are classified by $\mr{H}^{2}\big(\mr{X},\C\brak^{*}\big)$.
\end{remark}
\2
\begin{example}
 \cref{Poissonexample} shows that any DQ-algebroid on $\mr{X}$ induces a Poisson structure on $\mr{X}$. Conversely, it is a theorem of Kontsevich \cite{MR2062626} that in the $\mr{C}^{\infty}$ setting (locally for algebraic varieties) any Poisson structure is induced by some DQ-algebroid. The global algebraic quantisation is due to Yekutieli \cite{MR2183259} and Van den Bergh \cite{MR2344349}, and by Calaque et al. \cite{MR2364075} for complex manifolds.
\end{example}
\2
\begin{remark}
 If $\mathscr{A}_\mr{X}$ is a DQ-algebroid, the local notions of being locally free, coherent, flat, etc. make sense for an $\mathscr{A}_\mr{X}$-module $\mathscr{D}$.
\end{remark}

\2
\begin{definition}Let $\mathscr{R}$ be a sheaf of commutative $\C$-algebras.
\begin{itemize}
 \item An $\mathscr{R}$-algebroid is a $\C$-algebroid $\mathscr{A}$ together with a morphism of sheaves of $\C$-algebras $\mathscr{R} \to \mathscr{E}nd(\mr{id}_{\mathscr{A}})$.
 \item An $\mathscr{R}$-algebroid $\mathscr{A}$ is called invertible if $\mathscr{R}|_\mr{U} \to \mathscr{E}nd(\tau)$ is an isomorphism for every open $\mr{U}\subset \mr{X}$ and any $\tau \in \mathscr{A}(\mr U)$.
 \end{itemize}
\end{definition}
\2
$\hspace*{3mm}$Let $\iota : \C \to \C\brak$ be the canonical inclusion, define a $\C$-algebroid $\iota^{*}\mathscr{A}_\mr{X}$ by taking the stack associated with the prestack $\mathscr{B}$ given by $$\mathscr{B}(\mr{U}) = \mathscr{A}_\mr{X}(\mr{U}) \text{ and } \mr{Hom}_{\mathscr{B}(\mr{U})}(\sigma,\tau) = \mr{Hom}_{\mathscr{A}_\mr{X}(\mr{U})}(\sigma,\tau)/\hbar\mr{Hom}_{\mathscr{A}_\mr{X}(\mr{U})}(\sigma,\tau).$$ The so defined $\C$-algebroid is an ivertible $\mathscr{O}_\mr{X}$-algebroid. There are functors of $\C$-algebroids $$\mathscr{A}_\mr{X} \to \iota^* \mathscr{A}_\mr{X} \to \mathscr{O}_\mr{X}$$ and an equivalence of $\C$-algebroids $\iota^*\mathscr{A}_\mr{X} \simeq \mathscr{A}_\mr{X}/\hbar\mathscr{A}_\mr{X} \simeq \mathscr{O}_\mr{X}$. In particular, we get a functor preversing boundedness and coherence $$\iota^* :\mr{\mathbf{D}}(\mathscr{A}_\mr{X}) \to \mr{\mathbf{D}}(\iota^*\mathscr{A}_\mr{X}) \quad \iota^*:\mathscr{D} \mapsto \C\otimes_{\C\brak} \mathscr{D}.$$
$\hspace*{3mm}$The $\hbar$-localisation of a DQ-algebroid $\mathscr{A}_\mr{X}$ is $\mathscr{A}_\mr{X}^{\mr{loc}} = \C\cbrak\otimes _{\C\brak}\mathscr{A}_\mr{X}$. More generally, we have a functor $$\mr{loc}: \mr{\mathbf{D}}^{\mr b}(\mathscr{A}_\mr{X}) \to \mr{\mathbf{D}}^{\mr b}(\mathscr{A}_\mr{X}^{\mr{loc}}).$$
\begin{lemma}
 Let $\mathscr{D} \in \mr{\mathbf{D}}^{\mr b}_\mr{coh}(\mathscr{A}_\mr{X})$. Then we have $\mr{Supp}(\mathscr{D}) = \mr{Supp}(\iota^*\mathscr{D})$. In particular, $\mr{Supp}(\mathscr{D})$ is a closed analytic subset of $\mr{X}$.\\
 If $\mathscr{E} \in \mr{\mathbf{D}}^{\mr b}_\mr{coh}(\mathscr{A}_\mr{X}^{\mr{loc}})$, then $\mr{Supp}(\mathscr{E})$ is a closed analytic subset of $\mr{X}$, coisotropic for the Poisson structure defined by $\mathscr{A}_\mr{X}$.
\end{lemma}
\2
\begin{remark}
 Note that we do not have a global equivalence $\iota^*\mathscr{A}_\mr{X} \simeq \mathscr{O}_\mr{X}$ of invertible $\mathscr{O}_\mr{X}$-algebroids. This is generally only true locally. In the algebraic case, the vanishing of $\mr{H}^{2}(\mr{X},\mathscr{O}_\mr{X}^{*})$ in the Zariski topology implies that $\iota^*\mathscr{A}_\mr{X} \simeq \mathscr{O}_\mr{X}$ as $\mathscr{O}_\mr{X}$-algebroids globally.\\
 $\hspace*{3mm}$The second statement in the lemma is known as Gabber's theorem and doesn't hold for coherent $\mathscr{A}_\mr{X}$-modules - note that any closed analytic subset of $\mr{X}$ can be the support of such a module since any coherent $\mathscr{O}_\mr{X}$-module is a coherent $\mathscr{A}_\mr{X}$-module.
\end{remark}
\2
\begin{theorem}\label{perf}(Kashiwara-Schapira \cite{MR3012169})
 Let $\mr{X}$ be a complex manifold endowed with a DQ-algebroid $\mathscr{A}_\mr{X}$. Let $$\mathscr{D}, \mathscr{E} \in \mr{\mathbf{D}}^{\mr b}_\mr{coh}(\mathscr{A}_\mr{X})$$ and suppose that $\mr{Supp}(\mathscr{D})\cap \mr{Supp}(\mathscr{E})$ is compact. Then $\mr{RHom}_{\mathscr{A}_\mr{X}}(\mathscr{D},\mathscr{E})$ is a perfect complex of $\C\brak$-modules.
\end{theorem}
\2
\begin{definition}
 Let $\mr{X}$ be complex manifold endowed with a DQ-algebroid $\mathscr{A}_\mr{X}$, and let $\mr{Y}$ be a smooth submanifold of $\mr{X}$. A coherent $\mathscr{A}_\mr{X}$-module $\mathscr{D}$ supported on $\mr{Y}$ is called simple if $\iota^*\mathscr{D}$ is concentrated in degree $0$ and $\mr{H}^{0}(\iota^*\mathscr{D})$ is an invertible $\mathscr{O}_{\mr{Y}}\otimes_{\mathscr{O}_\mr{X}}\iota^*\mathscr{A}_\mr{X}$-module.
\end{definition}
\2
\begin{definition}
 Let $\mr{X}$ be a holomorphic symplectic variety equipped with a DQ-algebroid $\mathscr{A}_\mr{X}$.
 \begin{enumerate}
  \item An $\mathscr{A}_\mr{X}^{\mr{loc}}$-module is called holonomic if it is coherent and its support is a Lagrangian subvariety of $\mr{X}$.
  \item An $\mathscr{A}_\mr{X}$-module is called holonomic if it is coherent, $\hbar$-torsion free and its $\hbar$-localisation is holonomic.
  \item Let $\mr{L}$ be a smooth Lagrangian. An $\mathscr{A}_\mr{X}^{\mr{loc}}$-module $\mathscr{D}$ is called simple holonomic if there exists locally an $\mathscr{A}_\mr{X}$-module $\mathscr{D}^{0}$, simple along $\mr L$, which generates it, i.e. $(\mathscr{D}^{0})^{\mr loc} \simeq \mathscr{D}$.
 \end{enumerate}
\end{definition}
\2
\begin{theorem}\label{perv}(Kashiwara-Schapira \cite{10.2307/40068123})
 Let $\mr X$ be a holomorphic symplectic variety of dimension $2n$, equipped with a DQ-algebroid $\mathscr{A}_\mr{X}$. Suppose that $\mathscr{D}$ and $\mathscr{E}$ are two holonomic $\mathscr{A}_\mr{X}^{\mr loc}$-modules. Then the complex $\mr{R}\shom_{\mathscr{A}_\mr{X}^{\mr loc}}(\mathscr{D},\mathscr{E})[n]$ is a perverse sheaf.
\end{theorem}
\2
\begin{theorem}\label{existsimplehol}(D'Agnolo-Schapira \cite{MR2331247})
 Let $\mr{X}$ be a holomorphic symplectic variety and let $i : \mr{L} \xhookrightarrow{} \mr{X}$ be a smooth Lagrangian. Assume that the canonical bundle $\mr{K}_{\mr{L}}$ of $\mr{L}$ admits a square root. Then, for any choice of a square root $\mr{K}_\mr{L}^{1/2}$, there exists a simple $\hat{\mathscr{W}}_\mr{X}(0)$-module $\mathscr{D}_\mr{L}$, supported on $\mr{L}$, which quantises $\mr{K}_\mr{L}^{1/2}$.
\end{theorem}
\2
\begin{remark}
 Notation as in the previous theorem, consider the short exact sequence $$1 \to \C^{*} \to \mathscr{O}_{\mr{L}}^{*} \xrightarrow{\mr{dlog}} \mr{d}\mathscr{O}_{\mr{L}}\to 0.$$ It induces a long exact sequence in cohomology, and we are interested in the folloing part of it: $$\mr{H}^{1}(\mr{L},\C^{*}) \to \mr{H}^1(\mr{L},\mathscr{O}_{\mr{L}}^{*}) \xrightarrow{\alpha} \mr{H}^{1}(\mr{L},\mr{d}\mathscr{O}_{\mr{L}}) \xrightarrow{\delta} \mr{H}^{2}(\mr{L},\C^{*}).$$
$\hspace*{3mm}$Let $\C_{\mr{K}_\mr{L}^{1/2}}$ be the $\C$-algebroid associated to the class $\delta(\frac{1}{2}\alpha(\mr{c}_1(\mr{K}_\mr{L})))$. Then a general version of the above theorem asserts that there exists a simple $\hat{\mathscr{W}}_\mr{X}(0)\otimes \C_{\mr{K}_\mr{L}^{1/2}}$-module along $\mr{L}$, i.e. in general we only get a twisted $\hat{\mathscr{W}}_\mr{X}(0)$-module.\\
 $\hspace*{3mm}$Notice that $\C_{\mr{K}_\mr{L}^{1/2}}$ is trivial iff there exists a line bundle $\mathscr{L}$ such that $\mr{K}_\mr{L}\otimes \mathscr{L}^{\otimes 2}$ admits a flat connection, hence $\mathscr{L}$ can be quantised. In particular, this agrees with the results of \cite{MR3437831} since $\hat{\mathscr{W}}_\mr{X}(0) \simeq \hat{\mathscr{W}}_\mr{X}(0)^{\mr{op}}$ implies that the Atiyah class $\mr{At}(\hat{\mathscr{W}}_\mr{X}(0),\mr{L})=0$.
\end{remark}
\2
\begin{theorem}\label{pervref}
 Let $\mr X$ be a holomorphic symplectic variety of dimension $2n$, equipped with the canonical DQ-algebroid $\hat{\mathscr{W}}_\mr{X}(0)$. Suppose that $\mr{L}$ and $\mr{M}$ are smooth Lagrangians and assume that $\mr{K}_{\mr{L}}^{1/2}$ and $\mr{K}_\mr{M}^{1/2}$ exist. Let $\mathscr{D}^{0}_{\mr{L}}$ and $\mathscr{D}^{0}_{\mr{M}}$ be two simple holonomic $\hat{\mathscr{W}}_\mr{X}(0)$-modules, supported on $\mr{L}$ and $\mr{M}$, respectively, as in \cref{existsimplehol}. Then we have an isomorphism of perverse sheaves $$\mr{R}\shom_{\hat{\mathscr{W}}_\mr{X}}\big(\mathscr{D}_{\mr{L}},\mathscr{D}_{\mr{M}}\big)[n] \xrightarrow{\sim} \C\cbrak \otimes_{\C} \mathscr{P}_{\mr{L,M}},$$ where $\mathscr{D}_{\mr{L}}$ is the $\hbar$-localisation $\C\cbrak \otimes_{\C\brak}\mathscr{D}^{0}_{\mr{L}}$ and similarly for $\mathscr{D}_{\mr{M}}$.
\end{theorem}

\section{Applications}\label{sec3}

$\hspace*{3mm}$We begin with a few standard results on calculations of local $\ext$ sheaves and their multiplicative structure on locally complete intersections. In the second and third paragraphs we prove degeneration of the spectral sequences for a single Lagrangian and a pair of cleanly intersecting Lagrangians, respectively. We conclude with the formality of the endomorphism dg algebra $\mr{RHom}\big(i_*\mr{K}_\mr{L}^{1/2},i_*\mr{K}_\mr{L}^{1/2}\big)$ and the dg module $\mr{RHom}\big(i_*\mr{K}_\mr{L}^{1/2},j_*\mr{K}_\mr{M}^{1/2}\big)$ over it.
\paragraph{Sheaves on locally complete intersections.}
\begin{proposition}Let $i \rm : Z \xhookrightarrow{} X$ be a locally complete intersection. 
 Suppose $c = \mr{codim(Z,X)}$, and let $\mathscr{F}$ be a coherent sheaf on $\rm Z$. Then $$\mathscr{H}^{-i}(i^{*}i_{*}\mathscr{F})\cong \begin{cases}
 \mathscr{F}\otimes \wedge^{i}\mathscr{N}^{\vee}_{\mr{Z/X}}, \, 0\le i \le c\\
 0, \text{ otherwise.}                                                                                                                                          \end{cases}
$$ 
\end{proposition}
\2
\begin{proposition}\label{prop2}
 Let $i \rm : Z \xhookrightarrow{} X$ be a locally complete intersection of codimension $c$. Let $\mathscr{F}$ and $\mathscr{G}$ be coherent sheaves on $Z$. 
 \begin{enumerate}
  \item Assume $\mathscr{F}$ locally free, then we have $\ext^{i}(i_{*}\mathscr{F},i_{*}\mathscr{G}) \cong \begin{cases} i_{*}(\wedge^{i}\mathscr{N}_{\mr{Z/X}} \otimes \mathscr{F}^{\vee}\otimes \mathscr{G}), \, 0\le i \le c\\
  0, \text{ otherwise. }
  \end{cases}$
  \item The Yoneda product coincides with the usual cup product. More precisely, let $\mathscr{F}$, $\mathscr{G}$ be locally free sheaves, $\mathscr{H}$ any coherent sheaf, then the Yoneda multiplication $$\ext^{i}(i_{*}\mathscr{G},i_{*}\mathscr{H}) \otimes \ext^{j}(i_{*}\mathscr{F},i_{*}\mathscr{G}) \to \ext^{i+j}(i_{*}\mathscr{F},i_{*}\mathscr{H})$$ corresponds under the above isomorphisms to $$i_{*}(\wedge^{i}\mathscr{N}_{\mr{Z/X}} \otimes \mathscr{G}^{\vee}\otimes \mathscr{H}) \otimes i_{*}(\wedge^{j}\mathscr{N}_{\mr{Z/X}} \otimes \mathscr{F}^{\vee}\otimes \mathscr{G}) \to i_{*}(\wedge^{i+j}\mathscr{N}_{\mr{Z/X}} \otimes \mathscr{F}^{\vee}\otimes \mathscr{H}),$$ given by exterior product and the natural map $\mathscr{G}\otimes \mathscr{G}^{\vee} \to \mathscr{O}_\mr{Z}$.
 \end{enumerate}
\end{proposition}

\paragraph{Degeneration of the spectral sequence.}
Let $(\mr{X},\sigma)$ be a holomorphic symplectic variety. Recall that a subvariety $\mr L$ is called Lagrangian if $\sigma|_\mr{L}=0$ and $\mr{dim}\mr{L} = \frac{1}{2}\mr{dim}\mr X$. If $i:\mr{ L \xhookrightarrow{} X}$ is a smooth Lagrangian we have $\mathscr{T}_\mr{X} \cong \Omega^{1}_\mr{X}$ via the symplectic form, hence $i^{*}\mathscr{T}_\mr{X} \cong i^{*}\mr{\Omega^{1}_{X}}$. There is a commutative diagram: 
\[
\begin{tikzcd}
{}
0\arrow{r}
&\mathscr{T}_\mr{L}\arrow{r}\arrow{d}
&i^{*}\mathscr{T}_\mr{X}\arrow{r}\arrow{d}
&\mathscr{N}_\mr{L/X}\arrow{r}\arrow{d}
&0\\
0\arrow{r}
&\mathscr{N}^{\vee}_\mr{L/X}\arrow{r}
&i^{*}\mr{\Omega^{1}_{X}}\arrow{r}
&\mr{\Omega^{1}_{L}}\arrow{r}
&0
\end{tikzcd}
\]
which shows we have isomorphisms $\Omega^{q}_\mr{L} \cong \wedge^{q}\mr{\mathscr{N}_{L/X}}$. Hence the second page of the local-to-global $\mr{Ext}$ spectral sequence, in the Lagrangian case, is $\mr{E}_{2}^{p,q} = \mr{H}^{p}(\mr{L},\Omega^{q}_\mr{L})$.
\2
\begin{theorem}\label{ksrmk}
 Let $\mr{X}/\C$ be holomorphic symplectic, and consider a compact Kähler Lagrangian $i : \mr{L \xhookrightarrow{} X}$ whose canonical bundle admits a square root. Then the local-to-global $\mr{Ext}$ spectral sequence  $$\mr{E}_{2}^{p,q}= \mr{H}^{p}(\mr{L},\Omega^{q}_\mr{L}) \Rightarrow \mathrm{Ext}^{p+q}\big(i_{*}\mr{K}_\mr{L}^{1/2},i_{*}\mr{K}_\mr{L}^{1/2}\big)$$ degenerates on the second page. Hence $\mr{H}^{k}(\mr{L}/\C) = \oplus_{p,q}\mr{H}^{p}(\mr{L},\Omega^{q}_\mr{L}) = \mr{Ext}^k\big(i_{*}\mr{K}_\mr{L}^{1/2},i_{*}\mr{K}_\mr{L}^{1/2}\big)$.
 \end{theorem}
 \begin{proof}
  This proof was envisaged by Thomas, and Petit helped us make the initial sketch rigorous. It will be enough to show that $\mr{dim}_{\C}\big(\mr{Ext}^{i}\big(i_{*}\mr{K}_\mr{L}^{1/2},i_{*}\mr{K}_\mr{L}^{1/2}\big)\big) \ge \mr{dim}_{\C}\big(\mr{H}^{i}(\mr{L}/\C)\big)$.\\
  $\hspace*{3mm}$Let $\mathscr{A}_\mr{X}$ be the canonical quantisation $\hat{\mathscr{W}}_{\mr{X}}(0)$ of $\mr{X}$.
  We fix a square root $\mr{K}_\mr{L}^{1/2}$ of the canonical bundle. There exists a simple $\mathscr{A}_\mr{X}$-module $\mathscr{D}^{0}_\mr{L}$ on $\mr{L}$ quantising $\mr{K}_{\mr{L}}^{1/2}$. Let $\mathscr{A}^{\mr{loc}}_\mr{X}$ be the localisation $\C(\!(\hbar)\!) \otimes_{\C[\![\hbar]\!]}\mathscr{A}_\mr{X}$. Then $\mathscr{D}^{0}_\mr{L}$ localises to a simple holonomic DQ-module $\mathscr{D}_\mr{L}$ over $\mathscr{A}^{\mr{loc}}_\mr{X}$.\\
  $\hspace*{3mm}$\cref{pervref} implies that $$\mr{R}\shom\big(\mathscr{D}_\mr{L},\mathscr{D}_\mr{L}\big) \simeq \C(\!(\hbar)\!)_{\mr{L}},$$ so we get $\mr{Ext}^{i}_{\mathscr{A}^{\mr{loc}}_\mr{X}}\big(\mathscr{D}_\mr{L},\mathscr{D}_\mr{L}\big) = \mr{H}^{i}\big(\mr{L},\C(\!(\hbar)\!)\big)$. Hence the universal coefficients theorem implies that $$\mr{dim}_{\C\cbrak}\big(\mr{Ext}^{i}_{\mathscr{A}^{\mr{loc}}_\mr{X}}\big(\mathscr{D}_\mr{L},\mathscr{D}_\mr{L}\big)\big) = \mr{dim}_{\C}\big(\mr{H}^{i}(\mr{L}/\C)\big).$$ 
Furthermore \cref{perf}, which requires $\mr{L}$ compact, states that $$\mr{RHom}_{\mathscr{A}_\mr{X}}\big(\mathscr{D}^{0}_\mr{L},\mathscr{D}^{0}_\mr{L}\big) \in \mr{Perf}\big(\mr{Spec}\big(\C[\![\hbar]\!]\big)\big).$$ Then we can apply the semicontinuity theorem on $\C[\![\hbar]\!]$ to get that $$\mr{dim}_{\C}\big(\mr{H}^{i}\big(\C\otimes_{\C[\![\hbar]\!]}\mr{RHom}_{\mathscr{A}_\mr{X}}\big(\mathscr{D}^{0}_\mr{L},\mathscr{D}^{0}_\mr{L}\big)\big)\big) \ge \mr{dim}_{\C\cbrak}\big(\mr{H}^{i}\big(\C(\!(\hbar)\!)\otimes_{\C[\![\hbar]\!]}\mr{RHom}_{\mathscr{A}_\mr{X}}\big(\mathscr{D}^{0}_\mr{L},\mathscr{D}^{0}_\mr{L}\big)\big)\big).$$
  $\hspace*{3mm}$It's enough to observe that there is a quasi-isomorphism $$\C\otimes_{\C[\![\hbar]\!]}\mr{RHom}_{\mathscr{A}_\mr{X}}\big(\mathscr{D}^{0}_\mr{L},\mathscr{D}^{0}_\mr{L}\big) \simeq \mr{RHom}\big(i_{*}\mr{K}_\mr{L}^{1/2},i_{*}\mr{K}_\mr{L}^{1/2}\big)$$ and the projection formula implies that $$\C(\!(\hbar)\!)\otimes_{\C[\![\hbar]\!]}\mr{RHom}_{\mathscr{A}_\mr{X}}\big(\mathscr{D}^{0}_\mr{L},\mathscr{D}^{0}_\mr{L}\big) \simeq \mr{RHom}_{\mathscr{A}^{\mr{loc}}_\mr{X}}\big(\mathscr{D}_\mr{L},\mathscr{D}_\mr{L}\big).$$ Thus $\mr{dim}_{\C}\big(\mr{Ext}^{i}\big(i_{*}\mr{K}_\mr{L}^{1/2},i_{*}\mr{K}_\mr{L}^{1/2}\big)\big) \ge \mr{dim}_{\C\cbrak}\big(\mr{Ext}^{i}_{\mathscr{A}^{\mr{loc}}_\mr{X}}\big(\mathscr{D}_\mr{L},\mathscr{D}_\mr{L}\big)\big)=\mr{dim}_{\C}\big(\mr{H}^{i}(\mr{L}/\C)\big)$.
  \end{proof}
  
\paragraph{Degeneration in case of pairs of Lagrangians.}Having dealt with the case of one Lagrangian, we now turn to pairs of Lagrangians. Let $i: \mr{L} \xhookrightarrow{} \mr{X}$ and $j:\mr{M} \xhookrightarrow{} \mr{X}$ be smooth submanifolds. We need a few standard results computing $\mathscr{E}xt$ sheaves on smooth intersections.
\2
\begin{proposition}
 Assuming $\mr{L}\cap \mr{M}$ smooth, we have $$\ext^{p}(i_{*}\mathscr{O}_{\mr{L}},j_{*}\mathscr{O}_{\mr{M}}) \cong \wedge^{p-c} \mathscr{N} \otimes \mr{det}\mathscr{N}_{\mr{\mr{L}\cap\mr{M}/M}},$$ where $c= \mr{rk}\mathscr{N}_{\mr{L\cap M/M}}$, $\mathscr{N} \coloneqq \left.\mathscr{T}_\mr{X}\right|_{\mr{L}\cap \mr{M}}/(\mr{\left.\mathscr{T}_{L}\right|_{L\cap M}+\left.\mathscr{T}_{M}\right|_{L\cap M}})$ is the excess normal bundle.
\end{proposition}
\2
$\hspace*{3mm}$In the Lagrangian case, which we assume from now on, we have an exact sequence: $$\mr{0 \to \mathscr{T}_{\mr{L}\cap\mr{M}} \to \left.\mathscr{T}_{L}\right|_{\mr{L}\cap\mr{M}} \oplus \left.\mathscr{T}_{M}\right|_{\mr{L}\cap\mr{M}} \to \left.\mathscr{T}_{X}\right|_{\mr{L}\cap\mr{M}} \to \Omega_{\mr{L}\cap\mr{M}}} \to 0,$$ hence 
 $$\ext^{p}(i_{*}\mathscr{O}_{\mr{L}},j_{*}\mathscr{O}_{\mr{M}}) \cong \Omega^{q-c}_{\mr{L}\cap\mr{M}} \otimes \mr{det}\mathscr{N}_{\mr{\mr{L}\cap\mr{M}/M}}.$$
 The adjunction formula yields an isomorphism $$\mr{det}\mathscr{N}_{\mr{\mr{L}\cap\mr{M}/M}} \cong \left.\mr{K}_{\mr{M}}^{\vee}\right|_{\mr{L}\cap\mr{M}}\otimes \mr{K}_{\mr{L}\cap \mr{M}},$$ hence by \cref{lemmasmoothcase} we obtain $$\mr{det\mathscr{N}_{\mr{L}\cap\mr{M}/L}\otimes det\mathscr{N}_{\mr{L}\cap\mr{M}/M}} \cong \mathscr{O}_{\mr{\mr{L}\cap\mr{M}}}.$$
 \begin{corollary}\label{impcor}
 Assuming there exist squre roots $\mr{K_{L}^{1/2}}$ and $\mr{K_{M}^{1/2}}$, define the orientation bundle $$\mathscr{K}_{\mr{or}} \coloneqq \left(\mr{ \left.K_{L}^{1/2}\right|_{\mr{L}\cap\mr{M}}\otimes \left.K_{M}^{1/2}\right|_{\mr{L}\cap\mr{M}}}\right)^{\vee} \otimes \mr{K_{\mr{L}\cap\mr{M}}}$$ and set $n=\mr{codim(L,X)}$, so $c=n-\mr{dim(\mr{L}\cap\mr{M})}$. Then $$\ext^{p}\big(i_{*}\mr{K_{L}^{1/2}},j_{*}\mr{K_{M}^{1/2}}\big) \cong \Omega^{q-c}_{\mr{\mr{L}\cap\mr{M}}} \otimes \mathscr{K}_{\mr{or}}.$$
 \end{corollary}
 \begin{remark}\label{rmk}
 Notice that the line bundle $\mathscr{K}_{\mr{or}}$ is torsion, in fact of order $2$. Hence the monodromy representation associated to the local system $\mathfrak{K}_{\mr{or}}$, arising from $\mathscr{K}_{\mr{or}}$, is unitary - see beginning of next proof for a brief sketch.
 \end{remark}
 \2
  \begin{theorem}\label{igbcor}
 Let $\mr{X}/\C$ be a holomorphic symplectic variety. Suppose that $i : \rm L \xhookrightarrow{} X$, $j : \mr{M \xhookrightarrow{} X}$ are smooth Lagrangians with a compact Kähler intersection $\mr{L}\cap \mr{M}$. Assume that $\mr{K}_{\mr{L}}^{1/2}$ and $\mr{K}_\mr{M}^{1/2}$ exist. Then the $\mr{Ext}$ local-to-global spectral sequence  \begin{equation*}\label{eqq}\mathrm{E}_{2}^{p,q}=\mr{H}^{p}(\mr{L}\cap \mr{M},\Omega^{q-c}_{\mr{L}\cap \mr{M}} \otimes \mathscr{K}_{\mr{or}})\Rightarrow \mathrm{Ext}^{p+q}\big(i_{*}\mr{K_{L}^{1/2}},j_{*}\mr{K_{M}^{1/2}}\big)\end{equation*} degenerates on the second page. In particular, $$\mr{Ext}^{k}\big(i_{*}\mr{K_{L}^{1/2}},j_{*}\mr{K_{M}^{1/2}}\big) = \oplus_{p,q} \mr{H}^{p}(\mr{L}\cap \mr{M},\Omega^{q-c}_{\mr{L}\cap \mr{M}} \otimes \mathscr{K}_{\mr{or}}) = \mr{H}^{k-c}(\mathfrak{K}_\mr{or}).$$
\end{theorem}
\begin{proof}
 We have that $\mr{E}_{2}^{p,q} = \mr{H}^{p}(\mr{L}\cap \mr{M},\Omega^{q-c}_{\mr{L}\cap \mr{M}} \otimes \mathscr{K}_{\mr{or}})$, where $c$ is the codimension of $\mr{L}\cap\mr{M}$ in $\mr{L}$ and $\mr{M}$, and $\mathscr{K}_{\mr{or}}$ is defined in \cref{impcor}. The line bundle $\mathscr{K}_\mr{or}$ is $2$-torsion, hence admits a flat Chern connection, so the associated representation of $\pi_{1}(\mr{L}\cap \mr{M}, \mr{pt})$ is unitary. As a consequence the Hodge-to-de Rham spectral sequence degenerates on $\mr{E}_1$, so, analogously to the case of one Lagrangian, it will be enough to show that $$\mr{dim}_{\C}\big(\mr{Ext}^{i}\big(i_{*}\mr{K}_\mr{L}^{1/2},j_{*}\mr{K}_\mr{M}^{1/2}\big)\big) \ge \mr{dim}_{\C}\big(\mr{H}^{i-c}(\mathfrak{K}_\mr{or})\big).$$ 
 $\hspace*{3mm}$Let $\mathscr{A}_\mr{X}$ be the canonical quantisation $\hat{\mathscr{W}}_{\mr{X}}(0)$ of $\mr{X}$.
  We shall fix square roots $\mr{K}_\mr{L}^{1/2}$ and $\mr{K}_\mr{M}^{1/2}$ of the canonical bundles of $\mr{L}$ and $\mr{M}$. Let $\mathscr{D}^{0}_\mr{L}$ and $\mathscr{D}^{0}_\mr{M}$ be the simple $\mathscr{A}_\mr{X}$-modules on $\mr{L}$ and $\mr{M}$ quantising $\mr{K}_{\mr{L}}^{1/2}$ and $\mr{K}_{\mr{M}}^{1/2}$, respectively. Then we let $\mathscr{D}_\mr{L} = \C\cbrak \otimes_{\C\brak} \mathscr{D}^{0}_\mr{L}$ and $\mathscr{D}_\mr{M} = \C\cbrak\otimes_{\C\brak}\mathscr{D}^{0}_\mr{M}$ be the corresponding $\hbar$-localisations.\\
$\hspace*{3mm}$Since the intersection $\mr{L}\cap\mr{M}$ is smooth, \cref{pervref} implies that $$\mr{R}\shom_{\mathscr{A}_\mr{X}^{\mr{loc}}}\big(\mathscr{D}_\mr{L},\mathscr{D}_\mr{M}\big) \simeq \C(\!(\hbar)\!)_{\mr{L}\cap \mr{M}}\otimes_{\C}\mathfrak{K}_{\mr{or}}[-c],$$ so we conclude that $\mr{Ext}^{i}_{\mathscr{A}^{\mr{loc}}_\mr{X}}\big(\mathscr{D}_\mr{L},\mathscr{D}_\mr{M}\big) = \mr{H}^{i-c}(\mr{L}\cap\mr{M},\C(\!(\hbar)\!)\otimes_{\C}\mathfrak{K}_{\mr{or}})$. Hence the universal coefficients theorem implies that $$\mr{dim}_{\C\cbrak}\big(\mr{Ext}^{i}_{\mathscr{A}^{\mr{loc}}_\mr{X}}\big(\mathscr{D}_\mr{L},\mathscr{D}_\mr{M}\big)\big) = \mr{dim}_{\C}\big(\mr{H}^{i-c}(\mr{L}\cap\mr{M},\mathfrak{K}_{\mr{or}})\big).$$ Moreover \cref{perf}, which requires $\mr{L}\cap\mr{M}$ compact, states that $$\mr{RHom}_{\mathscr{A}_\mr{X}}\big(\mathscr{D}^{0}_\mr{L},\mathscr{D}^{0}_\mr{M}\big) \in \mr{Perf}\big(\mr{Spec}\big(\C[\![\hbar]\!]\big)\big).$$ Thus, we can apply the semicontinuity theorem on $\C[\![\hbar]\!]$ to conclude that $$\mr{dim}_{\C}\big(\mr{H}^{i}\big(\C\otimes_{\C[\![\hbar]\!]}\mr{RHom}_{\mathscr{A}_\mr{X}}\big(\mathscr{D}^{0}_\mr{L},\mathscr{D}^{0}_\mr{M}\big)\big)\big) \ge \mr{dim}_{\C\cbrak}\big(\mr{H}^{i}\big(\C(\!(\hbar)\!)\otimes_{\C[\![\hbar]\!]}\mr{RHom}_{\mathscr{A}_\mr{X}}\big(\mathscr{D}^{0}_\mr{L},\mathscr{D}^{0}_\mr{M}\big)\big)\big).$$ 
$\hspace*{3mm}$Now observe that there is a quasi-isomorphism $$\C\otimes_{\C[\![\hbar]\!]}\mr{RHom}_{\mathscr{A}_\mr{X}}\big(\mathscr{D}^{0}_\mr{L},\mathscr{D}^{0}_\mr{M}\big) \simeq \mr{RHom}\big(i_{*}\mr{K}_\mr{L}^{1/2},j_{*}\mr{K}_\mr{M}^{1/2}\big)$$ and the projection formula implies that $$\C(\!(\hbar)\!)\otimes_{\C[\![\hbar]\!]}\mr{RHom}_{\mathscr{A}_\mr{X}}\big(\mathscr{D}^{0}_\mr{L},\mathscr{D}^{0}_\mr{M}\big) \simeq \mr{RHom}_{\mathscr{A}^{\mr{loc}}_\mr{X}}\big(\mathscr{D}_\mr{L},\mathscr{D}_\mr{M}\big).$$ Hence $\mr{dim}_{\C}\big(\mr{Ext}^{i}\big(i_{*}\mr{K}_\mr{L}^{1/2},j_{*}\mr{K}_\mr{M}^{1/2}\big)\big) \ge \mr{dim}_{\C\cbrak}\big(\mr{Ext}^{i}_{\mathscr{A}^{\mr{loc}}_\mr{X}}\big(\mathscr{D}_\mr{L},\mathscr{D}_\mr{M}\big)\big)=\mr{dim}_{\C}\big(\mr{H}^{i-c}(\mr{L}\cap\mr{M},\mathfrak{K}_\mr{or})\big)$.
\end{proof}

\paragraph{Formality.}
 \begin{lemma}\label{freelemma}
 Let $\iota : \C \xhookrightarrow{} \C\brak$ be the inclusion of the central fibre and let $\mr{C} \in \mr{Perf}\big(\mr{Spec}\big(\C\brak\big)\big)$. Suppose that for all $i\in \mathbf{Z}$ we have $$\mr{dim}_\C\big(\mr{H}^{i}(\iota^{*}\mr{C})\big) = \mr{dim}_{\C\cbrak}\big(\mr{H}^{i}(\C\cbrak\otimes_{\C\brak} \mr{C})\big).$$
 Then the cohomology $\mr{H}(\mr{C})$ is free (of finite rank) over $\C\brak$.
 \end{lemma}
\begin{proof}
 Consider the exact triangle $\mr{C} \xrightarrow{\hbar} \mr{C} \to \iota^{*}\mr{C} \xrightarrow{} \mr{C}[1]$. It induces a long exact sequence in cohomology $$\mr{H}^{i}(\mr{C}) \xrightarrow{\hbar} \mr{H}^{i}(\mr{C}) \to \mr{H}^{i}(\iota^{*}\mr{C}) \to \mr{H}^{i+1}(\mr{C}) \xrightarrow{\hbar} \mr{H}^{i+1}(\mr{C}).$$ Hence there are exact sequences $$ 0 \to \C\otimes_{\C\brak}\mr{H}^{i}(\mr{C}) \to \mr{H}^{i}(\iota^{*}\mr{C}) \to \mr{Tor}_{1}^{\C\brak}\left(\C,\mr{H}^{i+1}(\mr{C})\right) \to 0.$$  In particular, we get that $\mr{dim}_{\C}\left(\C\otimes_{\C\brak}\mr{H}^{i}(\mr{C})\right) \le \mr{dim}_\C\left(\mr{H}^{i}(\iota^{*}\mr{C})\right).$
 Since $\mr{H}^{i}(\mr{C})$ is finitely generated, we may write it as $$\mr{H}^{i}(\mr{C}) = \C\brak^{d_i}\oplus \C\brak/\hbar^{k_1} \oplus \cdots \oplus \cdots \oplus\C\brak/\hbar^{k_{r_i}},$$ where $k_1,\cdots,k_{r_i},r_i \in \mathbf{N}$. Notice that $$\mr{dim}_{\C}\left(\C\otimes_{\C\brak}\mr{H}^{i}(\mr{C})\right) = d_i + r_i \text{ and } \mr{dim}_{\C\cbrak}\left(\C\cbrak\otimes_{\C\brak}\mr{H}^{i}(\mr{C})\right)=d_i.$$ It follows by flatness of $\C\cbrak$ that $$\mr{dim}_{\C\cbrak}\left(\C\cbrak\otimes_{\C\brak}\mr{H}^{i}(\mr{C})\right)= \mr{dim}_{\C\cbrak}\left(\mr{H}^{i}(\C\cbrak\otimes_{\C\brak} \mr{C})\right).$$ Hence, $r_i=0$ and $\mr{H}^{i}\left(\mr{C}\right)$ is free.
\end{proof}
\begin{proposition}
 Let $\mr{X}/\C$ be a holomorphic symplectic variety, endowed with its canonical DQ algebroid $\hat{\mathscr{W}}_{\mr{X}}(0)$. Suppose $\mr{L}$ is a smooth Lagrangian in $\mr{X}$ and let $\mathscr{D}_\mr{L}$ be a holonomic $\hat{\mathscr{W}}_{\mr{X}}$-module on $\mr{L}$. Then the quasi-isomorphism $\mr{R}\shom_{\hat{\mathscr{W}}_{\mr{X}}}\big(\mathscr{D}_{\mr{L}},\mathscr{D}_\mr{L}\big) \simeq \C\cbrak_\mr{L}$ is compatible with the multiplicative structures, i.e. it is a quasi-isomorphism of dg algebras.
\end{proposition}
\begin{proof}
 The question is local, we may assume $\mr{X}=\Omega_\mr{L}$ and $\mr{L}$ is the zero section in $\Omega_\mr{L}$. We shall fix coordinates $(z_1,\cdots,z_n)$ on $\mr{L}$. Further, since any two holonomic $\hat{\mathscr{W}}_{\mr{X}}$-modules are locally isomorphic, we may let $\mathscr{D}_{\mr{L}} = \mathscr{O}_\mr{L}\cbrak$. Then the Koszul complex $\mr{K}(\mathscr{O}_{\mr{L}}\cbrak,\partial_{1},\cdots,\partial_{n})$, associated with the coregular sequence $\partial_{1},\cdots,\partial_{n}$ acting (on the left) on $\mathscr{O}_{\mr{L}}\cbrak$, gives a multiplicative model for $\mr{R}\shom_{\hat{\mathscr{W}}_{\mr{X}}}\big(\mathscr{O}_{\mr{L}}\cbrak,\mathscr{O}_\mr{L}\cbrak\big)$ and the natural quasi-isomorphism $\C\cbrak_\mr{L} \to \mr{K}(\mathscr{O}_{\mr{L}}\cbrak,\partial_{1},\cdots,\partial_{n})$ is multiplicative.
\end{proof}

 \begin{theorem}\label{310}
  Let $\mr{X}/\C$ be holomorphic symplectic and let $i:\mr{L \xhookrightarrow{} X}$ be a smooth compact Kähler Lagrangian submanifold whose canonical bundle admits a square root. Then the differential graded algebra $\mr{RHom}\big(i_{*}\mr{K}_\mr{L}^{1/2},i_{*}\mr{K}_\mr{L}^{1/2}\big)$ is formal, in fact, quasi-isomorphic to the de Rham algebra $\mr{H}(\mr{L}/\C)$.
 \end{theorem}
 \begin{proof}  
  As before, let $\mathscr{A}_\mr{X}$ be the canonical quantisation $\hat{\mathscr{W}}_{\mr{X}}(0)$ of $\mr{X}$.
  We fix a square root $\mr{K}_\mr{L}^{1/2}$ of the canonical bundle. There exists a simple $\mathscr{A}_\mr{X}$-module $\mathscr{D}^{0}_\mr{L}$ on $\mr{L}$ quantising $\mr{K}_{\mr{L}}^{1/2}$. Let $\mathscr{A}^{\mr{loc}}_\mr{X}$ be the localisation $\C(\!(\hbar)\!) \otimes_{\C[\![\hbar]\!]}\mathscr{A}_\mr{X}$ and denote $\mathscr{D}_\mr{L} = \C\cbrak \otimes_{\C\brak} \mathscr{D}^{0}_{\mr{L}}$ the $\hbar$-localised simple holonomic DQ-module over $\mathscr{A}^{\mr{loc}}_\mr{X}$.\\ 
  $\hspace*{3mm}$Consider the differential graded algebra $\mr{RHom}_{\mathscr{A}_\mr{X}}\big(\mathscr{D}^{0}_\mr{L},\mathscr{D}^{0}_\mr{L}\big)$ over $\C[\![\hbar]\!]$.  By \cref{pervref}, we have that  $$\mr{R}\shom_{\mathscr{A}^{\mr{loc}}_\mr{X}}\big(\mathscr{D}_\mr{L},\mathscr{D}_\mr{L}\big) \simeq \C(\!(\hbar)\!)_{\mr{L}}.$$ Degeneration of the spectral sequence on $\hbar=0$ allows us to apply \cref{freelemma}, so the cohomology of $\mr{RHom}_{\mathscr{A}_\mr{X}}\big(\mathscr{D}^{0}_\mr{L},\mathscr{D}^{0}_\mr{L}\big)$ is free of finite rank over $\C\brak$. In particular, $\mr{Ext}_{\mathscr{A}_\mr{X}}\big(\mathscr{D}^{0}_\mr{L},\mathscr{D}^{0}_\mr{L}\big)$ is a formal deformation of the graded algebra $\mr{Ext}\big(i_*\mr{K}_\mr{L}^{1/2},i_*\mr{K}_\mr{L}^{1/2}\big)$ to the de Rham algebra $\mr{H}(\mr{L}/\C)\cbrak$. It follows that 
  \begin{equation}\label{ineq1}
  \mr{dim}_{\C}\big(\mr{HH}^{p,q}\big(\mr{H}(\mr{L}/\C)\big)\big) \le \mr{dim}_{\C}\big(\mr{HH}^{p,q}\big(\mr{Ext}\big(i_*\mr{K}_\mr{L}^{1/2},i_*\mr{K}_\mr{L}^{1/2}\big)\big)\big).
  \end{equation} Furthermore, collapse of the spectral sequence on $\hbar=0$ also implies that there exists a filtration $\mr{F}$ on the graded algebra $\mr{Ext}\big(i_*\mr{K}_\mr{L}^{1/2},i_*\mr{K}_\mr{L}^{1/2}\big)$ such that $$\mr{Gr}_\mr{F}\big(\mr{Ext}\big(i_*\mr{K}_\mr{L}^{1/2},i_*\mr{K}_\mr{L}^{1/2}\big)\big) \cong \mr{H}(\mr{L}/\C),$$ as graded algebras, so the associated (completed) Rees algebra is a formal deformation of the graded algebra $\mr{H}(\mr{L}/\C)$ to $\mr{Ext}\big(i_*\mr{K}_\mr{L}^{1/2},i_*\mr{K}_\mr{L}^{1/2}\big)\otimes_{\C}\C\cbrak.$ The Rees deformation gives the opposite of \eqref{ineq1}, hence $$\mr{dim}_{\C}\big(\mr{HH}^{p,q}\big(\mr{H}(\mr{L}/\C)\big)\big) = \mr{dim}_{\C}\big(\mr{HH}^{p,q}\big(\mr{Ext}\big(i_*\mr{K}_\mr{L}^{1/2},i_*\mr{K}_\mr{L}^{1/2}\big)\big)\big)$$ and by \cref{freelemma} the Hochschild cohomology groups $\mr{HH}^{p,q}\left(\mr{Ext}_{\mathscr{A}_\mr{X}}\big(\mathscr{D}^{0}_\mr{L},\mathscr{D}^{0}_\mr{L}\big)\right)$ are free over $\C\brak$ for all $p,q \in \mathbf{Z}$. As a consequence, the Hochschild cohomology groups with compact supports of Lunts, $\mr{HH}_c^{i}\big(\mr{Ext}_{\mathscr{A}_\mr{X}}\big(\mathscr{D}^{0}_\mr{L},\mathscr{D}^{0}_\mr{L}\big)\big)$, are free over $\C\brak$. It follows, by formality of $\mr{L}$, that the differential graded algebra $\mr{RHom}_{\mathscr{A}^{\mr{loc}}_\mr{X}}\big(\mathscr{D}_\mr{L},\mathscr{D}_\mr{L}\big)$ is formal. Hence, $\mr{RHom}_{\mathscr{A}_\mr{X}}\big(\mathscr{D}^{0}_\mr{L},\mathscr{D}^{0}_\mr{L}\big)$ is formal by \cref{formalityalg}.\\
  $\hspace*{3mm}$In order to conclude the formality proof, we note that there is a quasi-isomorphism $$\mr{RHom}\big(i_*\mr{K}_\mr{L}^{1/2},i_*\mr{K}_\mr{L}^{1/2}\big) \simeq \C\otimes_{\C\brak}\mr{RHom}_{\mathscr{A}_\mr{X}}\big(\mathscr{D}^{0}_\mr{L},\mathscr{D}^{0}_\mr{L}\big).$$Since the cohomology of $\mr{RHom}_{\mathscr{A}_\mr{X}}\big(\mathscr{D}^{0}_\mr{L},\mathscr{D}^{0}_\mr{L}\big)$ is free over $\C\brak$, it follows that the dg algebra $$\mr{RHom}\big(i_*\mr{K}_\mr{L}^{1/2},i_*\mr{K}_\mr{L}^{1/2}\big)$$ is formal.\\
  $\hspace*{3mm}$To prove the last statement of the theorem, note that $\mr{Ext}_{\mathscr{A}_\mr{X}}\big(\mathscr{D}^{0}_\mr{L},\mathscr{D}^{0}_\mr{L}\big)$ is a generically constant deformation which leaves the dimension of $\mr{HH}^{2,0}$ unchanged. The classical results on deformations of \cite{MR389978} extend to the graded case, thus $\mr{Ext}_{\mathscr{A}_\mr{X}}\big(\mathscr{D}^{0}_\mr{L},\mathscr{D}^{0}_\mr{L}\big)$ must be the trivial deformation - we give a summary of the proof, referring to \cite[Section~3]{MR389978} for details (in the non-graded case).\\
  $\hspace*{3mm}$Let $\mr{HH} \coloneqq \mr{HH}\big(\mr{Ext}\big(i_*\mr{K}_\mr{L}^{1/2},i_*\mr{K}_\mr{L}^{1/2}\big)\big)$, $\mr{HH}_{\hbar} \coloneqq \mr{HH}\big(\mr{Ext}_{\mathscr{A}_\mr{X}}\big(\mathscr{D}^{0}_\mr{L},\mathscr{D}^{0}_\mr{L}\big)\big)$ and $\mr{HH}_{(\hbar)}$ is the localisation of $\mr{HH}_\hbar$ at $\hbar$. The reduction $\mr{mod}\,\hbar$ gives a map $\mr{HH}_{\hbar} \to \mr{HH}$ and since $\mr{HH}_\hbar$ is free over $\C\brak$, taking a section of the reduction map, we get $\mr{HH}_\hbar \cong \mr{HH}\brak$.\\
  $\hspace*{3mm}$We have an isomorphism of graded algebras $$\varphi_{\hbar} :\mr{Ext}_{\mathscr{A}_\mr{X}}\big(\mathscr{D}^{0}_\mr{L},\mathscr{D}^{0}_\mr{L}\big)\otimes_{\C\brak}\C\cbrak \cong \mr{H}(\mr{L}/\C)\otimes_{\C}\C\cbrak.$$ Write $m_{\hbar}$ and $m_{\mr{dR}}$ for the multiplications of $\mr{Ext}_{\mathscr{A}_\mr{X}}\big(\mathscr{D}^{0}_\mr{L},\mathscr{D}^{0}_\mr{L}\big)$ and $\mr{H}(\mr{L}/\C)$, respectively. Then by definition 
  \begin{equation}\label{mult}
    m_{\hbar} = \varphi_{\hbar}^{-1}m_{\mr{dR}}(\varphi_{\hbar},\varphi_{\hbar}).
  \end{equation}
 Differentiating \eqref{mult} with respect to $\hbar$ gives 
 \begin{equation}\label{mult2}
 m_{\hbar}' = -\mr{d}_{\hbar}(\varphi_{\hbar}^{-1}\varphi_{\hbar}'),
 \end{equation} 
 where $\mr{d}_{\hbar}$ is the Hochschild differential of the graded algebra $\mr{Ext}_{\mathscr{A}_\mr{X}}\big(\mathscr{D}^{0}_\mr{L},\mathscr{D}^{0}_\mr{L}\big)$, i.e. $[m_\hbar']=0$ in $\mr{HH}^{2,0}_{(\hbar)}$. Writing $$m_{\hbar} = m + m_{r}\hbar^{r} + \cdots,\, r \ge 1,$$ we see that the left side of \eqref{mult2} is $$rm_{r}\hbar^{r-1}+(r+1)m_{r+1}\hbar^{r} + \cdots.$$ Hence we deduce that $$m_{r}+(r+1)/r\cdot m_{r+1}\hbar + (r+2)/r\cdot m_{r+2}\hbar^{2} +\cdots$$ is a $\hbar$-torsion class in $\mr{HH}^{2,0}_\hbar$, lifting the class $[m_r] \in \mr{HH}^{2,0}$, so it must vanish. Thus $[m_r]=0$ too. If $m_r =\mr{d}(\psi^{r})$, then we can kill $m_r$ using $\mr{id}-\psi^{r}\hbar^r$. By induction we get $\psi^{i}$ for all $i \ge r$, the infinite composition  $$\psi_\hbar = \big(\big(\mr{id}-\psi^{r}\hbar^r\big)\circ \big(\mr{id}-\psi^{r+1}\hbar^{r+1}\big)\circ \cdots \big)$$ makes sense and we have $\psi_\hbar^{-1}m_\hbar(\psi_\hbar,\psi_\hbar) = m$, showing the triviality of $m_\hbar$.\\
  $\hspace*{3mm}$Since $\mr{Ext}\big(i_*\mr{K}_\mr{L}^{1/2},i_*\mr{K}_\mr{L}^{1/2}\big)$ and $\mr{H}\left(\mr{L}/\C\right)$ are finite dimensional over $\C$ and become isomorphic after extending scalars to $\C\cbrak$, they are already isomorphic over $\C$.
 \end{proof}
 
\begin{corollary}
 Let $\mr{X}/\C$ be holomorphic symplectic, and consider a compact Kähler Lagrangian $i:\mr{L \xhookrightarrow{} X}$ such that $\mr{K}_\mr{L}^{1/2}$ exists. Then $\mr{Ext}\big(i_*\mr{K}_\mr{L}^{1/2},i_*\mr{K}_\mr{L}^{1/2}\big) \cong \mr{H}(\mr{L}/\C)$ as graded algebras.
\end{corollary}
\2
\begin{theorem}
 Let $\mr{X}/\C$ be holomorphic symplectic. Suppose that $i : \rm L \xhookrightarrow{} X$, $j : \mr{M \xhookrightarrow{} X}$ are compact Kähler Lagrangians with a smooth intersection. Asuume their canonical bundles admit square roots. Then $\mr{RHom}\big(i_{*}\mr{K_{L}^{1/2}}, j_{*}\mr{K_{M}^{1/2}}\big)$ is a formal differential graded module over the (formal) differential graded algebra $\mr{RHom}\big(i_{*}\mr{K_{L}^{1/2}},i_{*}\mr{K_{L}^{1/2}}\big)$. Moreover, we have a quasi-isomorphism of pairs $$\big(\mr{RHom}\big(i_{*}\mr{K_{L}^{1/2}},i_{*}\mr{K_{L}^{1/2}}\big),\mr{RHom}\big(i_{*}\mr{K_{L}^{1/2}}, j_{*}\mr{K_{M}^{1/2}}\big)\big)\simeq \big(\mr{H}(\mr{L}/\C),\mr{H}^{*-c}(\mr{L}\cap\mr{M},\mathfrak{K}_\mr{or})\big),$$ where $c$ is the codimension of $\mr{L}\cap\mr{M}$ in $\mr{L}$.
\end{theorem}
\begin{proof}
 Let $\mathscr{A}_\mr{X}$ be the canonical quantisation $\hat{\mathscr{W}}_{\mr{X}}(0)$ of $\mr{X}$.
  We shall fix square roots $\mr{K}_\mr{L}^{1/2}$ and $\mr{K}_\mr{M}^{1/2}$ of the canonical bundles of $\mr{L}$ and $\mr{M}$. Let $\mathscr{D}^{0}_\mr{L}$ and $\mathscr{D}^{0}_\mr{M}$ be simple $\mathscr{A}_\mr{X}$-modules on $\mr{L}$ and $\mr{M}$ quantising $\mr{K}_{\mr{L}}^{1/2}$ and $\mr{K}_{\mr{M}}^{1/2}$, respectively. Then we let $\mathscr{D}_\mr{L} = \C\cbrak \otimes_{\C\brak} \mathscr{D}^{0}_\mr{L}$ and $\mathscr{D}_\mr{M} = \C\cbrak\otimes_{\C\brak}\mathscr{D}^{0}_\mr{M}$ be the corresponding $\hbar$-localisations.\\
$\hspace*{3mm}$The complex $\mr{RHom}_{\mathscr{A}_\mr{X}}\big(\mathscr{D}^{0}_\mr{L},\mathscr{D}^{0}_\mr{M}\big)$ is a dg module over $\mr{RHom}_{\mathscr{A}_\mr{X}}\big(\mathscr{D}^{0}_\mr{L},\mathscr{D}^{0}_\mr{L}\big)$. Since the intersection $\mr{L}\cap\mr{M}$ is smooth, \cref{pervref} implies that $$\mr{R}\shom_{\mathscr{A}_\mr{X}^{\mr{loc}}}\big(\mathscr{D}_\mr{L},\mathscr{D}_\mr{M}\big) \simeq \C(\!(\hbar)\!)_{\mr{L}\cap \mr{M}}\otimes_{\C}\mathfrak{K}_{\mr{or}}[-c].$$ It follows by \cref{igbcor} and \cref{freelemma} that the cohomology of $\mr{RHom}_{\mathscr{A}_\mr{X}}\big(\mathscr{D}^{0}_\mr{L},\mathscr{D}^{0}_\mr{M}\big)$ is free over $\C\brak$. In particular, we see that $\mr{Ext}_{\mathscr{A}_\mr{X}}\big(\mathscr{D}^{0}_\mr{L},\mathscr{D}^{0}_\mr{M}\big)$ is a formal deformation of the graded module $\mr{Ext}\big(i_{*}\mr{K_{L}^{1/2}},j_{*}\mr{K_{M}^{1/2}}\big)$ to $\mr{H}^{*-c}(\mr{L}\cap\mr{M},\mathfrak{K}_{\mr{or}})\cbrak$. It follows that 
\begin{equation*}
  \mr{dim}_{\C}\big(\mr{HH}^{p,q}\big(\mr{H}(\mr{L}/\C),\mr{H}^{*-c}(\mr{L}\cap \mr{M},\mathfrak{K}_\mr{or})\big)\big) \le \mr{dim}_{\C}\big(\mr{HH}^{p,q}\big(\mr{H}(\mr{L}/\C),\mr{Ext}\big(i_*\mr{K}_\mr{L}^{1/2},j_*\mr{K}_\mr{M}^{1/2}\big)\big)\big).
  \end{equation*}
  As in the previous theorem, degeneration of the spectral sequence on $\hbar=0$ gives a filtration $\mr{F}$ on $\mr{Ext}\big(i_*\mr{K}_\mr{L}^{1/2},j_*\mr{K}_\mr{M}^{1/2}\big)$ such that $$\mr{Gr}_\mr{F}\big(\mr{Ext}\big(i_*\mr{K}_\mr{L}^{1/2},j_*\mr{K}_\mr{M}^{1/2}\big)\big) \cong \mr{H}^{*-c}(\mr{L}\cap \mr{M},\mathfrak{K}_\mr{or})$$ as graded modules over $\mr{H}(\mr{L}/\C)$. The Rees deformation argument then implies the opposite the above inequality, so we conclude
  \begin{equation*}
  \mr{dim}_{\C}\big(\mr{HH}^{p,q}\big(\mr{H}(\mr{L}/\C),\mr{H}^{*-c}(\mr{L}\cap \mr{M},\mathfrak{K}_\mr{or})\big)\big) = \mr{dim}_{\C}\big(\mr{HH}^{p,q}\big(\mr{H}(\mr{L}/\C),\mr{Ext}\big(i_*\mr{K}_\mr{L}^{1/2},j_*\mr{K}_\mr{M}^{1/2}\big)\big)\big).
  \end{equation*}
Thus, by \cref{freelemma}, the Hochschild cohomology groups $$\mr{HH}^{p,q}\big(\mr{Ext}_{\mathscr{A}_\mr{X}}\big(\mathscr{D}^{0}_\mr{L},\mathscr{D}^{0}_\mr{L}\big),\mr{Ext}_{\mathscr{A}_\mr{X}}\big(\mathscr{D}^{0}_\mr{L},\mathscr{D}^{0}_\mr{M}\big)\big)$$ are free over $\C\brak$ for all $p,q \in \mathbf{Z}$.\\
$\hspace*{3mm}$We observed in the beginning of the proof of \cref{igbcor} that the local system $\mathfrak{K}_{\mr{or}}$ corresponds to a unitary representation of the fundamental group of $\mr{L}\cap\mr{M}$. Then, as explained on page \pageref{page2}, by a theorem of Deligne (see \cite{MR972343}, \cite{MR1179076}), we conclude that $\mr{RHom}_{\mathscr{A}^{\mr{loc}}_\mr{X}}\big(\mathscr{D}_\mr{L},\mathscr{D}_\mr{M}\big)$ is a formal dg module over the dg algebra $\mr{RHom}_{\mathscr{A}^{\mr{loc}}_\mr{X}}\big(\mathscr{D}_\mr{L},\mathscr{D}_\mr{L}\big)$. Hence $\mr{RHom}_{\mathscr{A}_\mr{X}}\big(\mathscr{D}^{0}_\mr{L},\mathscr{D}^{0}_\mr{M}\big)$ is a formal dg module over the dg algebra $\mr{RHom}_{\mathscr{A}_\mr{X}}\big(\mathscr{D}^{0}_\mr{L},\mathscr{D}^{0}_\mr{L}\big)$ by \cref{modformality}.\\
$\hspace*{3mm}$Next we note that there is a natural quasi-isomorphism of dg modules $$\mr{RHom}\big(i_*\mr{K}_\mr{L}^{1/2},j_*\mr{K}_\mr{M}^{1/2}\big) \simeq \C\otimes_{\C\brak}\mr{RHom}_{\mathscr{A}_\mr{X}}\big(\mathscr{D}^{0}_\mr{L},\mathscr{D}^{0}_\mr{M}\big),$$ associated with the quasi-isomorphism of dg algebras $$\mr{RHom}\big(i_*\mr{K}_\mr{L}^{1/2},i_*\mr{K}_\mr{L}^{1/2}\big) \simeq \C\otimes_{\C\brak}\mr{RHom}_{\mathscr{A}_\mr{X}}\big(\mathscr{D}^{0}_\mr{L},\mathscr{D}^{0}_\mr{L}\big).$$ It is now enough to recall that the cohomology of $\mr{RHom}_{\mathscr{A}_\mr{X}}\big(\mathscr{D}^{0}_\mr{L},\mathscr{D}^{0}_\mr{M}\big)$ is free over $\C\brak$, hence $\mr{RHom}\big(i_*\mr{K}_\mr{L}^{1/2},j_*\mr{K}_\mr{M}^{1/2}\big)$ is formal over $\mr{RHom}\big(i_*\mr{K}_\mr{L}^{1/2},i_*\mr{K}_\mr{L}^{1/2}\big)$.\\
$\hspace*{3mm}$For the last assertion, we observe that $\mr{Ext}_{\mathscr{A}_\mr{X}}\big(\mathscr{D}^{0}_\mr{L},\mathscr{D}^{0}_\mr{M}\big)$ is a generically constant deformation that leaves $\mr{HH}^{1,0}$ constant. Such a deformation must be trivial, the proof being similar to the one for deformations of algebras discussed in the proof of \cref{310}. Thus, $\mr{Ext}\big(i_*\mr{K}_\mr{L}^{1/2},j_*\mr{K}_\mr{M}^{1/2}\big)$ and $\mr{H}^{*-c}(\mr{L}\cap \mr{M},\mathfrak{K}_\mr{or})$ are isomorphic since they are finite dimensional over $\C$ and become isomorphic upon extending scalars to $\C\cbrak$.
\end{proof}

\end{document}